\documentclass{amsart}
\usepackage[dvips]{color}
\usepackage{amssymb}
\usepackage[active]{srcltx}
\usepackage{amsmath, amssymb, euscript,  mathrsfs, enumerate, amsthm, amsfonts, graphicx, xcolor}
\usepackage{hyperref}

\newcommand{\dm}{r^{n-1} dr}

\newcommand{\cO}{\mathcal O}

\newcommand{\R}{{\mathbb R}}

\newcommand{\e}{\epsilon}
\newcommand{\bu}{{\bar u}}
\newcommand{\ba}{{\bar a}}
\newcommand{\x}{{\bf x}}

\newcommand{\be}{\begin{equation}}
\newcommand{\bee}{\begin{equation*}}
\newcommand{\ee}{\end{equation}}
\newcommand{\eee}{\end{equation*}}
\newcommand{\bs}{\begin{split}}
\newcommand{\esp}{\end{split}}

\newcommand{\bbox}{ \Big ( \frac{d }{dt}-\Delta_{M_t} \Big )}

\newcommand{\cD}{{\mathcal D}}

\newtheorem{thm}{Theorem}[section]
\newtheorem{cor}[thm]{Corollary}
\newtheorem{lemma}[thm]{Lemma}
\newtheorem{prop}[thm]{Proposition}
\newtheorem{claim}{Claim}[section]

\newtheorem{defn}{Definition}[section]
\theoremstyle{remark}
\newtheorem{remark}{Remark}[section]
\numberwithin{equation}{section}

\theoremstyle{definition}

\newcommand{\bremark}{\begin{remark} \em}
\newcommand{\eremark}{\end{remark} }

\title{Uniqueness of entire graphs evolving by \\Mean Curvature flow}

\author{Panagiota Daskalopoulos}
\address{ {\bf P. Daskalopoulos:} Department of Mathematics, Columbia University,  New York, NY 10027, USA.}
\email{pdaskalo@math.columbia.edu}
\thanks{P.~Daskalopoulos is supported by the  NSF award  DMS-1900702}
\author{Mariel Saez}
\address{ {\bf M. Saez:}  Departamento de matem\'aticas, P. Universidad Cat\'olica de Chile,  Santiago, Chile }
\email{mariel@mat.uc.cl}
\thanks{ M. S\'aez is supported by the grant Fondecyt Regular 1190388}

\begin{document} 

\begin{abstract}
In this paper we study the uniqueness of graphical mean curvature flow with  locally Lipschitz initial data. 
We first prove  that rotationally  symmetric entire graphs are  unique,  without any further assumptions. Our methods also give an alternative simple proof of uniqueness in the one dimensional  case.  
 In the general case,   we establish the uniqueness of entire proper graphs that satisfy  a   uniform lower bound  on the second fundamental form. 
The latter result extends to initial conditions that are proper graphs over subdomains of  $ \mathbb{R}^n$. A consequence of our  result is  the uniqueness of convex
entire graphs, which allow us to prove that Hamilton's Harnack estimate holds for mean curvature flow solutions that are convex entire graphs. 

\end{abstract}

\maketitle 

\tableofcontents

\section{Introduction}

The evolution under {\em Mean curvature flow}  studies a  family of  immersions 
$F(\cdot, t)  : M^n \to \R^{n+1}$, $t \in (0,T)$,    of n-dimensional  hypersurfaces in $\R^{n+1}$ such that  
\be\label{eqn-MCF} 
\frac{\partial}{\partial t} F(p,t)  =  H(p,t)  \, \nu(p,t), \qquad p \in M^n
\ee
where $H(p, t)$ and  $\nu(p, t)$ denote the Mean curvature and  inward  pointing normal of the surface
$M_t:= F(M^n, t)$ at  the point $F(p,t)$. 

\smallskip

We will assume in this work  that  $M_t, t \in (0,T]$ is a {\em complete non-compact graph }  over a domain  $\Omega_t  \subset \R^n$ 
(if $\partial \Omega_0 \neq \emptyset$ then  $\partial \Omega_t$ will evolve by MCF, that is  in general  it will change in time).  
Then, the solution $M_t$ can be written as   $M_t=\{ (x,u(x,t))\, : \,x\in \Omega_t \}$ for a height function $u(x,t)$. In the  case where $\Omega=\R^n$ we will  
say that $M_t$ is an {\em entire graph}. 

\smallskip 
The {\em height function} 
$u$ satisfies 
 the following  the quasilinear parabolic initial value problem
 \begin{equation}\label{eqn-IVPD} 
\begin{cases}
u_t = \Big (\delta^{ij}-\frac{D^i u D^j u}{1+|Du|^2}\Big ) D_{ij}u, \qquad  &(x,t)\in \Omega_t  \times (0,T]\\
u(x,0)=u_0(x), \qquad &x \in \Omega_0
\end{cases}
\end{equation} 
where $M_0 := \{ (x,u_0(x))\, : \,x\in \Omega_0\}.$ 
Here we sum over repeated indices. In what follows, we will refer to this equation as {\it graphical mean curvature flow.}

\smallskip

Although the  Mean curvature flow (MCF)  has been  extensively studied in the compact case   from many points of view  (such as existence and regularity, weak solutions, singularities, the extension of the flow through the singularities, flow with surgery) not much has been done in the non-compact case
beyond  the  fundamental works by   Ecker-Huisken  \cite{EH1, EH2} which deal with graphs over $\R^n$ and the more recent work by the second author and Schn\"urer \cite{SS} which deals  with graphs over domains. 

\smallskip

The  works   by Ecker-Huisken \cite{EH1, EH2}  establish the existence and local a'priori estimates of the graphical MCF  over $\R^n$. 
Also, in \cite{EH1} the uniqueness  of graphical solutions is addressed in some special cases. 
The  results in  \cite{EH2} show  that in some sense the MCF on   entire graphs
behaves better than the  heat equation on $\R^n$, namely {\em an entire graph  
solution exists for  all times,  independently from the growth of the initial surface at infinity}. 
The initial entire graph is assumed to be  locally Lipschitz.  
 Methods of similar spirit as in \cite{EH2} are used by the second author  and  Schn\"urer in  \cite{SS} to establish the existence of  
MCF solutions which are complete non-compact graphs 
over domains $\Omega_t \subset \R^n$. Note that if  $\partial \Omega_0 \neq \emptyset$ then  $\partial \Omega_t$ will evolve by MCF, that is in general  it will change in time. 

\smallskip 

While the works  \cite{EH1, EH2}  and \cite{SS} completely   address the existence of classical solutions to the graphical MCF 
with Lipschitz continuous initial data (on $\R^n$ or domains),  the uniqueness question in such generality has remained on open question. 
While the methods in \cite{EH1, EH2} imply that polynomial growth at infinity is preserved by the flow, the question of uniqueness is not addressed in those works.
  In \cite{CY} the authors address uniqueness of graphs in general ambient manifolds and high co-dimension. However, their result requires a uniform bound on the second fundamental form for all times. 
Our goal in this work is to address the {\em uniqueness of classical  solutions to }  \eqref{eqn-IVP} under  {\em minimal  assumptions}  on  the behavior of the initial data 
$u_0(x)$  as $|x| \to +\infty$,  and under {\em no assumptions} on the behavior of the solutions at infinity.

\smallskip 

We will first describe our results in the case of {\em entire graphs},
these are Theorems \ref{thm-CSF}-\ref{thm-general}. We will then state our  result in the case of domains. 

\smallskip 
For the reader's convenience let us state the following  {\em  existence result } for graphical  MCF over $\R^n$ that follows from the Ecker-Huisken works  \cite{EH1, EH2}: {\em 
Assume that  $u_0\colon \R^n \to\R$ is 
  a locally Lipschitz continuous function.  
 Then there exists a solution $u \colon  \R^n \times (0, \infty) \to \R$ of the initial value problem
   \begin{equation}\label{eqn-IVP} 
\begin{cases}
u_t = \Big (\delta^{ij}-\frac{D^i u D^j u}{1+|Du|^2}\Big ) D_{ij}u, \qquad  &(x,t)\in \R^n \times (0,T)\\
u(x,0)=u_0(x), \qquad &x \in \R^n
\end{cases}
\end{equation}  \eqref{eqn-IVP} 
with $T=+\infty$ which  is continuous up to $t=0$ and $C^\infty$-smooth for
  $t>0$.}

\smallskip

The striking feature of the result above is that existence holds for {\em any locally Lipschitz entire graph  initial data } that is {\em independently from the spatial growth of the
initial data $u_0(x)$, as $|x| \to +\infty$}. This is in contrast with the {\em heat equation} on $\R^n$ where existence is guaranteed only for initial data with at most quadratic exponential growth at infinity. The underlying reason for this difference is that the diffusion coefficient    $g^{ij} = \delta^{ij}-\frac{D^i u D^j u}{1+|Du|^2}$ in this
non-linear problem becomes small in a maximal direction  of the gradient where  $|Du| \to +\infty$. This behavior can simply be observed in the one-dimensional case of an entire graph 
$u \colon  \R \times (0, \infty) \to \R$ evolving by curve shortening  flow (CSF), where $u(x,t)$ satisfies the equation 
\be\label{eqn-CSF}
u_t = \frac{u_{xx}}{1+u_x^2}
\ee
or in  higher dimensions under rotational symmetry where $x_{n+1} = u(r,t)$, $r=|x|$ evolves by
\be\label{eqn-uradial}
u_t = \frac{u_{rr}}{1+ u_r^2} + \frac{n-1}r \, u_r.
\ee

Note that a  similar phenomenon has been observed for quasilinear equations of the form 
\be\label{eqn-pm} u_t = \Delta u^m, \qquad \mbox{on} \,\, \R^n \times (0, \infty)
\ee
in the range of exponents $ \frac{(n-2)_+}n < m < 1$ (see in \cite{DKb, HP} and the references therein).  In all  cases above the {\em slow diffusion}   at spatial infinity  when $|Du|  \to +\infty$ in \eqref{eqn-CSF} and \eqref{eqn-uradial},  or $u \to +\infty$ 
\eqref{eqn-pm} prevents instant blow up of  solutions with large growing  initial data as $|x| \to +\infty$. 

\smallskip  
 We will see that in the {\em one-dimensional}  case 
of the CSF (equation \eqref{eqn-CSF})  or the {\em rotationally symmetric} case of MCF (equation \eqref{eqn-uradial})
{\em uniqueness holds for any  entire graph solution  independently of its growth at infinity. } This is in sharp contrast with the {\em heat equation} in any dimension. 
More precisely we will show the following two results. The first shows the uniqueness of entire graph solutions to CSF:

\begin{thm}[Uniqueness of solutions to CSF]  \label{thm-CSF}
Let $u_1, u_2 \colon \R  \times (0, T] \to \R$, $T >0$ be   two smooth   solutions   of equation \eqref{eqn-CSF}  with the same 
Lipschitz continuous initial data $u_0$, that is   $\lim_{t \to 0} u_1(\cdot, t) = \lim_{t \to 0} u_2(\cdot, t)=u_0$. Then,  $u_1=u_2$ on $\R \times (0,T]$. 
 \end{thm}

The second, shows the uniqueness of rotationally symmetric entire graph solutions of MCF:

\begin{thm} [Uniqueness of rotationally symmetric MCF solutions] \label{thm-RS}
Let $u_1, u_2 \colon \R^n \times (0, T] \to \R$, $T >0$  be  two entire graph rotationally symmetric smooth solutions   of \eqref{eqn-IVP}  with the same 
Lipschitz continuous initial data $u_0(x)$, that is   $\lim_{t \to 0} u_1(\cdot, t) = \lim_{t \to 0} u_2(\cdot, t)=u_0$. Then,  $u_1=u_2$ on $\R^n \times (0,T]$.  \end{thm}

\smallskip 

We remark that Theorem \ref{thm-CSF}  is already covered by the results in \cite{CZ0}. However, we provide here a simpler and more direct proof in the case of entire one-dimensional graphs, in particular pointing out the similarity with fast-diffusion. 

\smallskip 

Regarding the {\em general case } of proper entire graphs,  we  establish the uniqueness  under a suitable lower bound on the second fundamental  form  
which prevents large oscillations of the  solution 
in different directions. We then extend this result to proper graphs over subdomains $\Omega \subset \R^n$. We begin by recalling the following definition.

\smallskip 

\begin{defn}[Proper graphs  over   subdomains $\Omega \subset \R^n$] \label{defn-proper} A graph $M:= \{ (x, u(x)) : x \in \Omega \}$ over a subdomain $\Omega \subset \R^n$
 defined by the height function  $u: \Omega \to \R$  is called proper if $u(x) \to +\infty$ as $x \to \partial  \Omega$ or $|x| \to +\infty$ $($the latter is assumed if  $\Omega$ is unbounded, in particular 
 when  $\Omega=\R^n )$. 
\end{defn}

\smallskip 

Let  $M_t =
\{ (x, u(x, t)): \, x \in \R^n \} $, $t \in (0,T)$   be  a proper entire graph solution to  mean curvature flow   \eqref{eqn-MCF} starting at $M_0$, 
which is  defined by the height function $u \colon \R^n \times (0,T) \to \R$. 
We denote  by  $v = \langle e_{n+1}, \nu \rangle^{-1}$  the {\em gradient function}
of  $M_t$, where $\nu$ denotes   the inward pointing unit normal on $M_t$. 
Since  $M_t$, $t \in (0,T)$ is assumed to be an  entire graph, 
$\langle e_{n+1}, \nu \rangle$ has always the same sign. Furthermore, our  assumption that $M_t$ is proper,   guarantees that 
\be v = \langle e_{n+1}, \nu \rangle^{-1} >0, \qquad \mbox{on} \,\, M_t, \,\, t \in [0, T]
\ee
in which case $v = \sqrt{1+|Du|}$.  In our result below we will further  assume that $M_t$ satisfies the lower bound  curvature condition  
\be\label{eqn-cond-curv}
v \, h^i_j  \geq -c \,  \delta_j^i,  \qquad \mbox{on} \,\, M_t, \, t \in (0,T]
\ee
for some uniform constant $c >0$.  Here $h^i_j $ is the second fundamental form and in the particular case of  graphs corresponds to $h^i_j=  \Big (\delta^{il}-\frac{D_i u D_l u}{1+|Du|^2}\Big ) \frac{D_{lj}u}
{\sqrt{1+|Du|^2}}.$ 

\smallskip 

Our uniqueness  result states as follows:

\begin{thm}[General uniqueness result for entire graphs] \label{thm-general} Assume that $u_0: \R^n \to \R$ is a locally Lipschitz function defining a proper  entire graph  $M_0=
\{ (x, u_0(x)): \, x \in \R^n \} \subset \R^{n+1}$.

Let $u_1, u_2 \colon \R^n \times (0,T] \to \R$ be two  smooth  solutions of \eqref{eqn-IVP} defining two  entire graph solutions $M_t^1 =
\{ (x, u_1(x, t)): \, x \in \R^n \} $ and $M_t^2 =\{ (x, u_2(x, t)): \, x \in \R^n \} $ of MCF \eqref{eqn-MCF} which both satisfy condition \eqref{eqn-cond-curv} and have  the same initial data $u_0$, that is  $\lim_{t \to 0} u_1(\cdot, t) =  \lim_{t \to 0} u_2(\cdot, t) =u_0$. Then, $u_1=u_2$ on $\R^n \times (0,T]$,
that is $M_t^1 = M_t^2$ for all $t \in (0,T]$. 
\end{thm}  

\smallskip

\begin{remark} \begin{enumerate}[i)] 
\item Theorem \ref{thm-general}  implies that uniqueness holds under convexity with no other growth conditions on the initial data (see in the last section \ref{sec-convex}). As a   consequence  Hamilton's differential Harnack inequality holds  for convex graphs evolving under Mean Curvature Flow (see Corollary \ref{cor:Harnack estimate} in Section \ref{sec-convex}). A related result was recently discussed in \cite{BLL} in the context of translating solutions.

\item Theorem \ref{thm-general} shows  that uniqueness  holds for initial data $u_0(x)$  which has  arbitrarily large  growth as $|x| \to +\infty$,  as long as  the lower curvature bound \eqref{eqn-cond-curv} holds.

\item Theorem \ref{thm-general} only assumes the  lower bound \eqref{eqn-cond-curv} in comparison  with the results in \cite{CY} which assume upper and lower bounds on the second fundamental form. 

\end{enumerate} 
\end{remark} 

\smallskip 

At last we will discuss the uniqueness for  {\em graphs over subdomains}  of $\mathbb{R}^n$. In that context, the result in \cite{SS} guarantees the existence of smooth solutions:
 {\em Let $\Omega_0 \subset\R^{n+1}$ be a bounded open set and $u_0\colon \Omega_0 \to\R$
  a locally Lipschitz continuous function with $u_0(x)\to\infty$ for
  $x\to x_0\in\partial \Omega_0$.  Then there exists $(\cD,u)$, where
  $\cD \subset\R^{n+1}\times[0,\infty)$ is relatively open, such
  that $u$ is a solutions of the  graphical mean curvature flow
  \begin{equation}\label{eqn-IVP-B} 
\begin{cases}
u_t = \Big (\delta^{ij}-\frac{D_i u D_j u}{1+|Du|^2}\Big ) D_{ij}u, \qquad  &(x,t)\in \cD \setminus(\Omega_0\times\{0\})\\
u(x,0)=u_0(x), \qquad &x \in \Omega_0
\end{cases}
\end{equation} 
 The function $u$ is smooth for
  $t>0$ and continuous up to $t=0$, $u(\cdot,0)=u_0$ in
  $\Omega_0$ and $u(x,t)\to\infty$ as $(x,t)\to\partial\cD$, where
  $\partial\cD$ is the relative boundary of $\cD$ in
  $\R^{n+1}\times[0,\infty)$.}\\
  
   It is relevant to remark that in this theorem the domain of definition for the function $u$ changes in time and it is given by the  mean curvature flow 
   evolution of $\partial \Omega_0$ (see the discussion in \cite{SS}). More precisely, $u(x,t)$ is a graph over  $ \Omega_t$ where $\Omega_t \times \{t\} =\cD \cap \big (\R^{n}\times \{t\} \big ) $  and $\partial \Omega_t$ agrees with the evolution by mean curvature flow  of $\partial \Omega_0$ at time $t$, provided that this evolution is smooth. In addition, it is possible to see from the proof in \cite{SS} that if $(x_k, t_k)\to (\bar{x}, \bar{t})\in \partial\cD$ and $|\bar{x}|\leq R$ for some $R>0$,  then $u(x_k, t_k)\to\infty$.
\smallskip    
%Note also that in  \cite{SS} the existence result is proven  for unbounded domains $\Omega$ under the additional condition  $u(x,t)\to \infty $ as $|x|\to \infty$ for every $t\geq 0$. In our uniqueness theorem we will weaken this assumption and we will only require $u_0$ to be bounded from below. 

\smallskip 

Our  uniqueness result for graphs over subdomains states as follows: 
   
\begin{thm}[General uniqueness result for subdomains] \label{thm-general-sub}

 Let $\Omega_0  \subset\R^{n}$ be an open  set such that $\partial \Omega_0$ has a unique smooth evolution by  Mean Curvature Flow in $(0,T]$ and let $\Omega_t$ be such that $\partial \Omega_t$ agrees with the evolution of  $\partial \Omega_0$  at time $t$. 
Assume that $u_0: \Omega_0  \to\mathbb{R}$ is a locally Lipschitz function defining a proper graph   $M_0=
\{ (x, u_0(x)): \, x \in \Omega_0 \} \subset\mathbb{R}^{n+1}$. 

Let $u_1, u_2 \colon \Omega_t \times (0,T] \to \R$ be two  smooth solutions of \eqref{eqn-IVP-B} defining two proper entire graph solutions $M_t^1 =
\{ (x, u_1(x, t)): \, x \in \Omega_t \} $ and $M_t^2 =\{ (x, u_2(x, t)): \, x \in \Omega_t  \} $ of MCF \eqref{eqn-MCF}, both  satisfying condition \eqref{eqn-cond-curv},  and having  the same initial data $u_0$,  that is  $\lim_{t \to 0} u_1(\cdot, t) =  \lim_{t \to 0} u_2(\cdot, t) =u_0$. Assume in addition that if
 $(x_k, t_k)\to (\bar{x}, \bar{t})\in \partial\cD$ and $|\bar{x}|\leq R$ for some $R>0$ then $u_i(x_k, t_k)\to\infty$.
  Then, $u_1=u_2$ on $\cD=\cup_{t\in[0,T]} \Omega_t\times\{t\}$,
that is $M_t^1 = M_t^2$ for all $t \in (0,T]$. 
\end{thm}

The organization of this paper is as follows: In Sections \ref{sec-CSF} and \ref{sec-RC} we give  the proofs of Theorems \ref{thm-CSF} and \ref{thm-RS}, respectively. Section \ref{sec-GC} 
is devoted to the proofs of Theorems \ref{thm-general}   and \ref{thm-general-sub}. Finally, Section \ref{sec-convex} is devoted to the proof of Hamilton's differential Harnack inequality.\\

We conclude this  section with  the  following remarks. 

\begin{remark} 
In Theorem \ref{thm-general-sub},  if the evolution of  $\partial \Omega_0$ is not unique, it follows from the proof of the result that for each  evolution $\Omega_t$ there is at most one proper graphical solution satisfying assumption \eqref{eqn-cond-curv}. 
%
%ii) The existence of  solutions satisfying the condition of  Theorem \ref{thm-general-sub}, follow from the result in \cite{SS}  only if $u_0$ is  proper. However, if $u_0\not \to \infty $ as $
%|x|\to \infty$ and $u_0>-C$, the existence proof  
% in \cite{SS} can still be performed in combination with the estimates of \cite{EH2}.
 \end{remark}

 \begin{remark}
 Uniqueness for other non-compact flows has been discussed in other works. For instance, 
uniqueness  results  for complete Ricci Flow  are discussed in \cite{CZ1} and  \cite{PT}.  The uniqueness for complete Yamabe flow in hyperbolic space is discussed in \cite{S}.
\end{remark}

 {\bf Acknowledgments.} We would like to thank S. Lynch and Jingze Zhu for their  helpful remarks. Also    M. Langford for bringing to our attention the question of a differential Harnack inequality  in this setting.

\section{Curve shortening flow - Theorem \ref{thm-CSF}} \label{sec-CSF}

In this section we will show that entire graph smooth solutions to Curve Shortening Flow (that is \eqref{eqn-IVP} for $n=1$ and $\Omega =\mathbb{R}$)
are unique without any growth assumptions at spatial infinity. This result is in contrast with the case of the heat equation where at most quadratic exponential growth 
at infinity is required for uniqueness. As mentioned in the introduction  Theorem \ref{thm-CSF}  is already covered by the results in \cite{CZ0}. We  provide here a simpler and more direct proof in the case of entire graphs.

The evolution of a curve  $y=u(x,t)$ on the plane is  given by ${\displaystyle  u_t = \frac{u_{xx}}{1+ u_x^2}}$ which can be also written in divergence form as 
\be\label{eqn-u1}
u_t = \big (\arctan (u_x) \big)_x .
\ee 
Differentiating in $x$ we see that  $v:=u_x$  satisfies the equation 
\be\label{eqn-v0}
v_t = \big (\arctan v  \big)_{xx}.
\ee 

The proof of Theorem \ref{thm-CSF} will be based on the following simple observation which we prove next.

\begin{lemma}\label{lemma-ineq}
For any $\gamma \in (0,1]$, the following holds 
$$(\arctan v_1 - \arctan v_2 )_+ \leq 2  \, ( v_1 - v_2)_+^\gamma, \qquad \forall v_1, v_2 \in [0, + \infty).$$
\end{lemma} 

\begin{proof}[Proof of Lemma \ref{lemma-ineq}] 
Fix a number $\gamma \in (0, 1]$. We may assume  that $v_1 > v_2$ and write
$$(\arctan v_1 - \arctan v_2 )_+ = \int_{v_2}^{v_1}  \frac 1{1+ s^2} \, ds.$$
Assume first that $v_1 > v_2 \geq 1$. In this case, for  any number $\gamma \in (0,1]$ we have $v_1 \geq (v_1-v_2)^{1-\gamma}$,
so that the above gives   
$$(\arctan v_1 - \arctan v_2 )_+ \leq  \int_{v_2}^{v_1}  \frac 1{s^2} \, ds =
\frac{v_1-v_2}{v_1v_2} \leq \frac{v_1-v_2}{v_1} \leq  \frac{v_1-v_2}{(v_1-v_2)^{1-\gamma}} \leq (v_1-v_2)_+^\gamma.$$
 In the case that $0 < v_2 < 1 < v_1$ we have 
 $$(\arctan v_1 - \arctan v_2 )_+  \leq \int_{v_2}^1 \, ds + \int_1^{v_1}  \frac 1{s^2} \, ds  
 \leq (1-v_2) + \frac{v_1-1}{v_1}  \leq 2\,  (v_1-v_2)_+^\gamma$$
 since for any $\gamma \in (0,1]$ we have $1-v_2 <  (1-v_2)^\gamma < (v_1-v_2)^\gamma $ and $\frac{v_1-1}{v_1} < \frac{v_1-v_2}{v_1} <  (v_1-v_2)^\gamma$.
 The last inequality follows from $v_1 > (v_1-v_2)^{1-\gamma}$ which holds in this case.  
Finally, for  $0 < v_2 < v_1 \leq 1$, we have 
$$(\arctan v_1 - \arctan v_2 )_+ \leq (v_1-v_2)_+ \leq (v_1-v_2)_+^{\gamma}.$$
\end{proof}

\smallskip

\begin{proof}[Proof of Theorem \ref{thm-CSF} ] 
The proof follows the method  by Herrero and Pierre in \cite{HP}. Let $v_1=u_{1x}$ and $v_2=u_{2x}$.  We will first show that $v_1 \equiv v_2$ on $\R  \times  [0,T)$. To this end, we set  $w=(v_1-v_2)_+$. 
Since $v_1, v_2$ satisfy equation \eqref{eqn-v},  Kato's inequality implies that $w$  satisfies the differential inequality 
\be\label{eqn-w0}
w_t \leq  (a w)_{xx}, \qquad  \mbox{on}\,\,  \R \times (0,T)
\ee 
in the sense of distributions, where 
\bee
a :=  \frac {(\arctan v_1 - \arctan v_2)_+}{(v_1-v_2)_+}.  
\eee
Our observation in Lemma \ref{lemma-ineq} shows that for any $\gamma \in (0,1)$ we have 
$$0 \leq a  \leq 2 \, w^{-1+\gamma}.$$
We will use that momentarily.

\smallskip

Consider the test function 
$\varphi(x)  = \psi ( \frac{x}{R} ) $
where  $\psi(\rho)$ is a  smooth cut-off function supported in  $(-2,2)$ such that $ 0 \leq \psi \leq 1$, $\psi(\rho)=1$ for $x \in [-1,1]$.  
Integrating the differential inequality \eqref{eqn-w0} against $\varphi$,   we obtain 
\be\label{eqn-ineq10}
\begin{split}
\frac {d}{dt}  \int w(\cdot, t)  &\, \varphi  \,  dx  \leq  \int ( a  w) (\cdot, t)   \,  \varphi''   \, dx, \qquad t \in (0,T).
\end{split} 
\ee
\smallskip 
For any number $\gamma \in (0,1)$ (to be fixed  at the end of our proof)  we  use   inequality  $0 \leq a \leq 2 \, w^{-1+\gamma}$
 to conclude 
\bee
\begin{split}
\frac {d}{dt}  \int w \, \varphi  \,  dx  &\leq  2 \,  \int w^\gamma \,  |\varphi '' | \,  dx \leq C  \, \Big ( \int w \varphi \,  dx \Big )^\gamma \Big ( \int  |\varphi''   |^{\frac1{1-\gamma}}  
\varphi^{-\frac \gamma {1-\gamma} }  \, dx  \Big )^{1-\gamma}.\\
\end{split}
\eee
Since $|\varphi'' (x)| \leq C R^{-2} |\psi''(\rho)|$, $x= R\, \rho$,   and $\psi $ is supported in the interval $[-2,2]$ we have 
$$\int  |\varphi'' |^{\frac1{1-\gamma}}   \varphi^{-\frac \gamma {1-\gamma} }  \, dx  \leq C R^{1- \frac 2{1-\gamma}}
\int    |\psi''  |^{\frac1{1-\gamma}}   \psi^{-\frac \gamma {1-\gamma} }  \, d\rho.$$
\smallskip
For any  $\gamma \in (0,1)$ we can choose  cutoff $\psi=\psi_\gamma$ such that $\int   |\psi''  |^{\frac1{1-\gamma}}   \psi^{-\frac \gamma {1-\gamma} }  \, d\rho \leq C_\gamma$. 
We conclude that  $I(t):=  \int w(\cdot, t)   \varphi  \,  d x$ satisfies 
$$I'(t) \leq C_\gamma  \, I(t)^\gamma \,   R^{-(1+\gamma)}.$$
Integrating the last inequality on $[0, \bar t]$ for any $\bar t \in (0, T)$ while using that $\lim_{t \to 0} I(t)=0$ (this follows from the fact that that $v_1(\cdot, 0)=v_2(\cdot,0)$
a.e.)  we  obtain
$$I(\bar t )^{1-\gamma}  \leq C_\gamma  \,  \bar t  \, R^{  - (1+\gamma)  }  \implies I(\bar t) \leq C_n \, \bar t^{\frac 1{1-\gamma}} \, R^{ - \frac {1+\gamma}{1-\gamma}}.$$
Finally recalling that $\varphi \equiv 1$ on $[-R, R]$, we get 
$$\int_{-R}^R (v_1-v_2)_+(x,t ) dx  \leq  C_n \, \bar t^{\frac 1{1-\gamma}} \, R^{ - \frac {1+\gamma}{1-\gamma}}.$$
Letting $R \to +\infty$  and using monotone convergence we conclude that
${\displaystyle \int_0^\infty (v_1-v_2)_+(x, \bar t )\,dx  =0}$, for all $t\in [0, T)$.  
Therefore,  conclude that $(v_1-v_2)_+  \equiv 0$ on $[0, \infty)  \times [0, t_0]$, i.e. $(u_1)_x(\cdot, t)  \leq (u_2)_x(\cdot, t) $  in $\R$. Similarly $(u_2)_x(\cdot, t)  \leq (u_1)_x
(\cdot, t)$   in $\R$
implying that for any $t \in [0,T)$, we have $(u_1)_x( \cdot,t)  =   (u_2)_x(\cdot,t) $  in $\R$. This and the fact that $u_1=u_2$ at time $t=0$ easily give us that 
$u_1 \equiv u_2$, finishing our proof. 
\end{proof}

\section{Rotationally symmetric solutions - Theorem \ref{thm-RS}}  \label{sec-RC}
In this section we will consider the uniqueness of  rotationally symmetric  solutions of the initial value problem \eqref{eqn-IVP} on $\R^n \times (0, T)$.  
On a  radial solution $u(r,t)$ the evolution equation in  \eqref{eqn-IVP} becomes \
\be\label{eqn-uradial2}
u_t = \frac{u_{rr}}{1+ u_r^2} + \frac{n-1}r \, u_r.
\ee 
 Differentiating \eqref{eqn-uradial2} with respect to $r$ we find that  the derivative  $v:=u_r$ of any solution $u$ of \eqref{eqn-uradial}
satisfies the equation 
\begin{equation}\label{eqn-v}
 v_t=(\arctan v)_{rr}+ \big (\frac{n-1}{r} v)_r. 
 \end{equation}
\smallskip

\begin{proof}[Proof of Theorem \ref{thm-RS} ] 
The proof follows the method    by Herrero and Pierre in \cite{HP} and is a generalization of the one-dimensional case with the necessary adaptations.  
We simply denote by  $u_1(r,t), u_2(r,t)$ the rotational symmetric profiles
 we let  $v_1=u_{1r}$ and $v_2=u_{2r}$.  Set $w=(v_1-v_2)_+$. Since, $v_1, v_2$ both satisfy \eqref{eqn-v}, Kato's inequality implies that $w:=(v_1-v_2)_+$ satisfies 
\be\label{eqn-w}
w_t \leq  \Delta (a w) -  \frac{n-1}{r} (a w)_r  +  \big (\frac {n-1}r  \,   w \big )_r
\ee 
in the sense of distributions, where 
\bee
a :=  \frac {(\arctan v_1 - \arctan v_2)_+}{(v_1-v_2)_+}.  
\eee
Similarly with the one-dimensional case, the crucial observation is that  for any $\gamma \in (0,1)$ we have 
$0 \leq a  \leq 2 \, w^{-1+\gamma}$.

\smallskip

Consider the test function 
$$\varphi_R(r,t)  = \psi \Big( \frac{r^2 + 2(n-1) t }{R^2} \Big) $$
where  $\psi(\rho)$ is a  smooth cut-off function defined on  $ [0, +\infty)$ such that $ 0 \leq \psi \leq 1$, $\psi(\rho)=1$ for $0 \leq \rho \leq 1$ and 
$\psi(\rho) \equiv 0$ for $\rho \geq 2$. Then,
$$(\varphi_R)_t =  \frac {2(n-1)}{R^2} \psi', \qquad (\varphi_R)_r = \frac {2r} {R^2} \psi'  \implies  (\varphi_R)_t = \frac{n-1}r \, (\varphi_R)_r$$
and
$$(\varphi_R)_{rr}  = \frac {4r^2} {R^4} \psi'' + \frac{2}{R^2} \, \psi' \implies \Delta \varphi_R = \frac {4r^2} {R^4} \psi'' + \frac{2n}{R^2} \, \psi' .$$
Hence, using  $ (\varphi_R)_t = \frac{n-1}r \, (\varphi_R)_r$, we obtain 
\bee
\begin{split}
\frac {d}{dt}  \int w &\, \varphi_R  \,  \dm  =  \int  w_t  \, \varphi_R  \,\dm  +  \int w  \, (\varphi_R)_t   \,\dm\\
\leq&  \int a \, w  \,  \Delta \varphi_R  \, \dm - \int \frac{n-1}{r} (a w)_r \, \varphi_R \, \dm  +  \int \big (\frac{n-1}{r} w \big )_r \, \varphi_R \, \dm
+  \int \frac{n-1}{r} w \, (\varphi_R)_r   \dm.
\end{split} 
\eee
Performing integration by parts on the second and third terms, using that 
\bee
 \int \frac{n-1}{r} (aw)_r \, \varphi_R \, \dm  = - \int \frac{n-1}{r} \, a w \, (\varphi_R)_r  \, \dm   - \int \frac{  (n-2)(n-1) }{r^2}\,    a w  \varphi_R  \, \dm
\eee
we obtain (after cancellations) that 
\be\label{eqn-ineq1}
\begin{split}
\frac {d}{dt}  \int w &\, \varphi_R  \,  d \mu   \leq  \int a  w  \,  \Delta \varphi_R  \, \dm + \int \frac{n-1}{r} \, a w \, (\varphi_R)_r \,  \dm\\
&  +  \int \frac{  (n-2)(n-1) }{r^2}\,    a w \,  \varphi_R  \, \dm  - \int \frac{(n-1)^2 }{r^2}   \,  w\,  \varphi_R   \, \dm.
\end{split} 
\ee
Next notice that 
$$a:=  \frac {(\arctan v_1 - \arctan v_2)_+}{(v_1-v_2)_+} = \frac 1{1 + \bar v^2}$$
for some $\bar v$ between $v_1$ and $v_2$, hence $a \leq 1$. It follows that 
$$ \int \frac{  (n-2)(n-1) }{r^2}\,    a w \,  \varphi_R  \, \dm  - \int \frac{(n-1)^2 }{r^2}   \,  w\,  \varphi_R   \, \dm \leq 
- \int \frac{n-1}{r^2} w\,  \varphi_R   \, \dm\leq 0.$$

\smallskip 
Let $\gamma \in (0,1)$ be any number (to be chosen at the end of our proof) and use  the  inequality   $0 \leq a \leq 2 \, w^{-1+\gamma}$ shown in 
Lemma \ref{lemma-ineq} 
 to bound the first two terms on the right hand side of \eqref{eqn-ineq1}. We conclude  that 
\bee
\begin{split}
\frac {d}{dt}  \int w \, \varphi_R  \,  \dm   &\leq  C \,  \int w^\gamma \big (   |\Delta \varphi_R  | +  |(n-1) \, r^{-1} \, (\varphi_R)_r|  \big ) \,  \dm\\
&\leq C  \, \Big ( \int w \varphi_R \,  \dm \Big )^\gamma \Big ( \int  \big (   |\Delta \varphi_R  | +  |r^{-1}  (\varphi_R)_r|  \big )^{\frac1{1-\gamma}} \, \varphi_R^{-\frac \gamma {1-\gamma} } \dm \Big )^{1-\gamma}.\\
\end{split}
\eee
Observing  that for $0 \leq t \leq t_0$ and  $R \gg 1$ large we have  
$$ |\Delta \varphi_R(r,t)   | +  |r^{-1}  (\varphi_R)_r(r,t) |  \leq C_n \, R^{-2} \big (  |\psi''(\rho)| + |\psi'(\rho) \big )$$
where $\rho:= \frac{r^2 + 2(n-1) t }{R^2}$ we get 
\bee
\begin{split}
\Big \{  \int  \Big (   |\Delta \varphi_R(r,t)  | &+  |r^{-1}  (\varphi_R)_r(r,t)|  \Big )^{\frac1{1-\gamma}}  \varphi_R(r)^{-\frac \gamma {1-\gamma} }  r^{n-1} dr \Big \}^{1-\gamma}\\
&\leq R^{-2} \Big \{  \int  \big (   |\psi''(\rho)| + |\psi'(\rho) \big )^{\frac1{1-\gamma}} \, \psi(\rho)^{-\frac \gamma {1-\gamma} }  r^{n-1}(\rho) \, dr(\rho)
\Big \}^{1-\gamma}
\end{split}
\eee
where $r^2(\rho) = R^2 \rho - 2(n-1) t$, which in particular implies    $r \, dr = \frac {R^2}2 \,   d\rho$. Thus,  
\bee
\begin{split}
\int  \big (   |\psi'' (\rho) &+ |\psi' (\rho)| \big )^{\frac1{1-\gamma}} \, \psi(\rho)^{-\frac \gamma {1-\gamma} }  r^{n-1}(\rho) \, dr(\rho)   \\&= 
\frac {R^2} {2} \,
 \int  \big (   |\psi'' (\rho) + |\psi' (\rho)| \big )^{\frac1{1-\gamma}} \, \psi(\rho)^{-\frac \gamma {1-\gamma} } \, \big (R^2 \rho - 2(n-1)t \big )^{\frac{n-2}2} \, d\rho \\
 &\leq   C_n \, R^n \int  \big (   |\psi'' (\rho) + |\psi' (\rho)| \big )^{\frac1{1-\gamma}} \, \psi(\rho)^{-\frac \gamma {1-\gamma} }  \, d\rho
\end{split}
\eee
where we have used that on the support of $\psi', \psi''$ where $\rho \leq 2$, and for $0 \leq t \leq t_0$ and $R \gg \max (1, t_0)$, one has 
$(R^2 \rho - 2(n-1)t)^{\frac{n-2}2} \leq  C_n \, R^{n-2}.$ 

\smallskip

For any  $\gamma \in (0,1)$ we can choose  cutoff $\psi=\psi_\gamma$ for which the support of  $\psi', \psi''$ lies in $[1,2]$ such that 
$$ \int_1^2  \big (   |\psi'' (\rho) + |\psi' (\rho)| \big )^{\frac1{1-\gamma}} \, \psi(\rho)^{-\frac \gamma {1-\gamma} }  \, d\rho \leq C(n,\gamma).$$
We then   conclude from the above discussion that $I(t):=  \int w  \varphi_R  \,  \dm$ satisfies 
$$I'(t) \leq C(n, \gamma) \, I(t)^\gamma \,   R^{-2+ n\, (1-\gamma)}.$$
Since $\gamma \in (0,1)$ can be any number, we may choose  $\gamma = \gamma(n) \in (0,1]$ so  that $n\, (1-\gamma) < 2 $,  and integrating the last inequality on $[0, \bar t]$ for any $\bar t \in (0, T)$ while using that $I(0)=0$, we  obtain
$$I(\bar t )^{1-\gamma}  \leq C_n \,  \bar t  \, R^{  -2 + n(1-\gamma)  }  \implies I(\bar t) \leq C_n \, \bar t^{\frac 1{1-\gamma}} \, R^{n - \frac 2{1-\gamma}}.$$

\smallskip 
Finally recalling that $\varphi_R \equiv 1$ on $[0, R]$, we get 
$$\int_0^R (v_1-v_2)_+(r,t )\dm   \leq  R^{n - \frac 2{1-\gamma}}.$$
Letting $R \to +\infty$, using that $n - \frac 2{1-\gamma} <0$,  and monotone convergence yields
$\int_0^R (v_1-v_2)_+(\cdot ,  t ) \, \dm  =0$, for  all $t\in [0, T)$.  
Therefore,  we conclude that $(v_1-v_2)_+  \equiv 0$  on  $\R^n  \times [0, T)$, i.e. $(u_1)_r \leq (u_2)_r$. Similarly $(u_2)_r \leq (u_1)_r$ a.e.  in $\R^n  \times [0, T)$
implying that $(u_2)_r \equiv   (u_1)_r$. This and the fact that $u_1\equiv u_2$ at time $t=0$ easily give us that $u_1 \equiv  u_2$  on $\R^n \times [0,t_0]$,
for all $t_0 <T$  finishing our proof. 
\end{proof}

\section{The general case} \label{sec-GC}

Our goal in this section is to give the proof of our general uniqueness results, Theorem \ref{thm-general} and Theorem \ref{thm-general-sub}. 
We will see that the proof of the latter  theorem  is almost identical to  the proof of the former. Hence, we will
omit most of the proof of Theorem \ref{thm-general-sub}, pointing out only the minor differences.

For the sake of completeness we  show next that for entire graphs the condition $u_0 \geq C$ is preserved under the flow, which implies that if  the initial condition  is a proper entire graph, then the solution is proper as well, uniformly in time. Both facts will be used our proofs.   Because we are dealing with  non-compact solutions,  we will use the localization techniques developed in \cite{EH2}. 

\begin{lemma} \label{lemma-llb}
Let $u$ be a solution to \eqref{eqn-IVP}  on $\R^n \times (0, T)$ and  assume that $u_0(x) \geq C$  on $|\x - \x_0| \leq R$, $\x = (x, u_0(x))$,  for some fixed point $\x_0 \in \R^{n+1}$ and 
some number $R >1$. 
Then,  we have 
 $$u(x,t)\geq C-\frac{10}{R}t$$
on the parabolic ball $ |\x-\x_0|^2+2nt\leq \frac{R^2}{2}$, $\x=(x, u(x,t))$ $($provided it   is non-empty$)$. 
 
In particular, if  $u_0\geq C$ on  $\mathbb R^n$,  then for every $t \in (0,T)$
  we have $u(\cdot, t) \geq C$ on 
  $\mathbb R^n$.

\end{lemma}

\begin{proof}
 We will do all calculations in geometric coordinates, that is we assume that  our solutions are given by the embedding $\x = F(p,t)$ as in \eqref{eqn-MCF} 
and we define  $$U_R(p,t):=(u -C)\Big (1-\frac{|\x-\x_0|^2+2nt}{R^2}\Big )_+ +  \frac{5}{R}\, t$$
where $u:= \langle F, e_{n+1} \rangle$ and $\x = F(p,t)$. 
Our  assumption $u_0 \geq C$ in $B_R(x_0)$, gives    $U_R\geq 0$ at $t=0$. Furthermore,  
$$(U_R)_t-\Delta U_R=-2\nabla u\cdot 2 \frac{(\x-\x_0)^T}{R^2}+ \frac{5}{R}\geq - \frac{4}{R} + \frac{5}{R}>0.$$
The maximum principle implies that $U_R$ does not have any interior minima and $U_R\geq 0$. In particular, if $|\x-\x_0|^2+2nt\leq \frac{R^2}2$ then
$$0\leq \frac{u-C}{2}+  \frac{5}{R}t,$$ and the first result follows.

In the case where $u_0\geq C$ globally on  $\mathbb R^n$, then  for any $x_0 \in \R^n$, $t \in (0, T)$, 
we apply the above result taking $\x_0 = (x_0, u_0(x))$ and choosing $R \gg 1$ so that $ |\x-\x_0|^2+2nt \leq \frac{R^2}{2}$
if $\x = (x_0, u(x_0,t))$.  We readily conclude that $u(x_0,t) \geq C - \frac{10}{R} \, t$ 
and by  taking $R\to \infty$ we obtain that $u(x_0, t) \geq C$. Since $x_0 \in \R^n$ and $t \in (0,T)$ are  arbitrary, the second result follows. 
\end{proof}

\begin{cor}\label{cor-proper}  Let $u$ be a solution to \eqref{eqn-IVP}  on $\R^n \times (0, T]$ and  assume that $\lim_{|x| \to +\infty}  u_0(x) = + \infty$. Then, we have
$$ \lim_{|x| \to +\infty}  u(x,t)  = + \infty, \qquad\mbox{uniformly in $t \in (0,T]$}.$$

\end{cor} 

\begin{proof} We  begin by observing that our assumption that   $\lim_{|x| \to +\infty} u_0(x) = +\infty$ implies that $u_0 \geq C$ for some $C \in \R$
and hence by the previous lemma, $u \geq C$ as well. 

Now, for  every $k \gg 1$ let  $R_k > k$ be a sufficiently large number so  that $u_0(x) \geq k$ for $|x| \geq R_k$. 
For any $x_0 \in \R^n$ such that $|x_0|> 4R_k$, let $\x_0 = (x_0,0)$. Then,
$$u_0(x) \geq  k, \qquad \mbox{on} \,\,  |\x-\x_0|   \leq 2 R_k, \,\, \x=(x, u_0(x))$$
and hence, by the  previous lemma, for any $t \in (0, T)$, we have 
$$u(x,t) \geq k - \frac{5}{R_k} t, \qquad \mbox{on} \,\,  |\x-\x_0|^2+2nt \leq 4 R_k^2,  \, \x=(x, u(x,t)).$$
We may choose $k, R_k  \gg 1$ so that $2nT < R_k^2$ and $\frac{5}{R_k} T <1$. 
Evaluating the above estimate at $\x = (x_0, u(x_0,t))$, for any $t \in (0, T)$, it gives us that 
$$u(x_0, t) \geq k-1, \qquad \mbox{provided } \,\, |\x-\x_0|= |u(x_0, t)|  \leq R_k.$$
We conclude that  for any $|x_0| \geq 4 R_k$  and $t \in (0, T)$ we  either have $u(x_0,t) \geq k-1$ or $|u(x_0, t)| \geq R_k$.
Since, $u \geq C$ (be our initial observation) and $R_k \geq k$, we conclude that in either case $u(x_0,t) \geq k-1$, for all $t \in (0,T)$
and all $|x_0| \geq 4 R_k$. Since, $R_k$ is independent of $t$, the result readily follows. 
 
\end{proof}

%{\color{red} Note that  when we take  $c=0$ in Condition \ref{eqn-cond-curv} this states    that convexity is preserved by the flow. Although this holds for the evolution of compact manifolds, it is not direct verify it in the non-compact setting. In the lemma below we show that provided a polynomial  growth conditions for the solution  (which is expected to be preserved by the flow from  the results in \cite{EH1}) we have that Condition \ref{eqn-cond-curv} is also preserved} 
 One may ask whether  condition \ref{eqn-cond-curv} is preserved in time, namely if  $v h_i^j \geq -c$ at time $t=0$ implies that  $v h_i^j \geq c $ for  $t >0$. 
Although this is easy to verify  for the evolution of compact manifolds,  in the  non-compact setting it becomes  challenging.  Actually, even the case where $c=0$ is not known to 
hold in the general graphical non-compact setting.  In the lemma below we show that the condition is preserved under a suitable polynomial growth condition on the solution (which is expected to be preserved by the flow from  the results in \cite{EH1}).

\begin{lemma}
Assume that $v h_i^j \geq -c$ at time $t=0$, for some constant $c >0$  and  that for all times we have  $| h^i_j \,v|\leq C \, |\x|^q$ and that $\frac{|\nabla v|}{v} \leq C|\x|$. Then,   \
Condition \ref{eqn-cond-curv}  holds for every $t\geq 0$.
\end{lemma}

\begin{proof}
Let $f^i_j = h^i_j \,v+c$. Then, following \cite{EH1}  we have in geometric coordinates that
%\be\label{eqn-fij}
%\bbox  f^i_j=|A|^2h^i_j+ c\,  |A|^2 v^{-1}  \delta_j^i = |A|^2 f^i_j.
%\ee
%Then it is easy to compute that 
$$ \bbox  f^i_j \,v=-\frac{2}{v}\langle \nabla ( f^i_j \,v), \nabla v\rangle.$$

Let $\gamma=|\x|^2+2nt+1$ and $p>q$ (for instance $p=q+1$) and define $F= e^{-Kt} \,\gamma^{-p}\,  f^i_j \,v$. From our assumption $F\to 0$ as $|x|\to \infty$. Assume that there is an interior minimum that is negative. Then

\begin{align*}0\geq \bbox F&= -\frac{2 e^{-Kt} \,\gamma^{-p}}{v}\langle \nabla ( f^i_j \,v), \nabla v\rangle-p(p+1) e^{-Kt} \,\gamma^{-p-2}  f^i_j \,v |x^T|^2\\&\,-2p  e^{-Kt} \,\gamma^{-p-1}
\langle \nabla ( f^i_j \,v), x^T\rangle-Ke^{-Kt} \,\gamma^{-p}\,  f^i_j \,v.\end{align*}

Observe that $|x^T|^2\leq \gamma$. Then at the interior critical point we have $\gamma^{-p} \nabla ( f^i_j \,v)= p \gamma^{-p}  f^i_j \,v x^T$ and 
$$0\geq  \bbox F\geq  F ( C-p(p+1)\gamma^{-1}+2p^2 \gamma^{-1}-K).$$

Since $\gamma\geq 1$, by choosing $K$ large enough (depending on $C$ and $p$) we have that the right hand side is positive when $F<0$ which is a contradiction.

\end{proof}

\begin{remark}
Note that $\frac{|\nabla v|}{v} \leq |A| v$. Then, the results in \cite{EH1, EH2} imply that if $|A| v\leq |\x|$ holds at $t=0$, then this is preserved in time and the condition of our lemma is met with $q=1$.
\end{remark}

\subsection{Proof of Theorem \ref{thm-general} } 
%under the additional assumption that $u_0(x) \to +\infty$, as $x \to +\infty$. }
\begin{proof}
To simplify the notation in this proof  we denote $u=u_1$ and $\bu = u_2$, that is we assume that 
  $u, \bu: \R^n \times (0,T] \to \R$ are  the two smooth solutions to  \eqref{eqn-IVP} with  initial data $u_0$ as in the statement of Theorem \ref{thm-general}. 
%(to simplify the notation we call ) 
%We assume in addition that $u(x,t) \to +\infty$ and  $\bu(x,t) \to +\infty$ as $|x| \to +\infty$, for all $t \in (0,T)$. 
Since $u_0$ is proper we have $u_0 \geq -C$ for some constant $C >0$. 
Hence, by adding on $u_0$  the constant $C+1$ we may assume without  out loss of generality that $u_0 \geq 1$. 
Lemma \ref{lemma-llb} implies that 
$$u, \bu  \geq 1, \qquad \mbox{on} \,\, \R^n \times (0,T].$$
To show that $\bu = u$,  it is sufficient to prove that $\bu \leq u$, since the same argument will also imply that $u \leq \bu$, thus showing that $u=\bu$.

The solutions $u, \bu$ satisfy equations 
$$u_t  = \left(\delta^{ij}-\frac{D_i u D_j u}{1+|Du|^2}\right) D_{ij}u, \qquad \bu_t  = \left(\delta^{ij}-\frac{D_i \bu D_j \bu}{1+|D\bu|^2}\right) D_{ij}\bu.$$
Set  $a_{ij} = \delta^{ij}-\frac{D_i u D_j u}{1+|Du|^2}$, $\ba_{ij} = \delta^{ij}-\frac{D_i \bu D_j \bu}{1+|D\bu|^2}$ 
and define 
$$w:= u - \bar u.$$
Then, subtracting the above equations, we find that the function $w$  satisfies the equation
\be\label{eqn-w}
w_t - a_{ij} D_{ij} w =  ( a_{ij} - \bar a_{ij}  )\, D_{ij} \bu
\ee

The main idea in the proof is to introduce the {\em supersolution} 
$$\zeta(x,t) := \e \, (t+\e)\, u^2(x,t)$$
 for any given $\epsilon >0$ small. At the end we will let $\e \to 0$. 
First, we use  $u_t - a_{ij} D_{ij} u =0$ and  find that $\zeta$ satisfies   
$$\zeta_t - a_{ij} D_{ij} \zeta = - 2 \e \, (t+\e)  \, a_{ij} D_j u D_i u + \e u^2,$$ 
where 
\bee
\begin{split}
a_{ij} D_j u D_i u &=  \left(\delta^{ij}-\frac{D_i u D_j u}{1+|Du|^2}\right) D_i u D_j u = \delta^{ij}D_i u D_j u  -\frac{(D_i u)^2 (D_j u)^2}{1+|Du|^2}\\
&= |Du|^2 \big ( 1- \frac{|Du|^2}{1+|Du|^2} \big ) = \frac{|Du|^2}{1+|Du|^2}.
\end{split}
\eee
Combining the above gives
\bee
\zeta_t - a_{ij} D_{ij} \zeta = - 2 \e \, (t+\e)  \, \frac{|Du|^2}{1+|Du|^2} + \e u^2 \geq \e \, ( u^2 - 2\, (t + \e)).
\eee
Since $u \geq 1$, we conclude that for $t \leq 1/4$ and $\e < 1/10$, we have 
\be\label{eqn-zeta}
\zeta_t - a_{ij} D_{ij} \zeta >  \frac \e2 u^2.
\ee

\smallskip 

Set next 
$$W := w - \zeta=u-\bu - \e \, (t +\e)  \, u^2.$$
By \eqref{eqn-w} and \eqref{eqn-zeta} we find that  $W$ satisfies 
\be\label{eqn-W} 
W_t - a_{ij} D_{ij} W <   ( a_{ij} - \ba_{ij}  )\, D_{ij} \bu -  \frac \e 2 u^2. 
\ee
Furthermore,  our assumption that $u =\bu$ at $t=0$ (in the sense that $\lim_{t \to 0} \big [ u(\cdot,t) -  \bu(\cdot,t) \big ] =0$) 
we have 
\be \label{eqn-W10}   \lim_{t \to 0} W(x,t)  = - \e^2 u(x,0)   \leq   - \e^2<0, \qquad \mbox{uniformly on any $K \subset  \R^n$ compact}.
\ee
(The uniform convergence on compact sets  follows from  the bounds  in \cite{EH2} which give us local bounds 
on the second fundamental form $|A|  \leq C/ \sqrt{t}$ for both solutions $u, \bu$ where $C$ depends on the initial data). 

\smallskip 

Let 
$$T^*= \min \Big (  T, \frac 14, \frac 1{10c} \Big )$$
where $c$ is the constant in \eqref{eqn-cond-curv}. 
We will use \eqref{eqn-W} -\eqref{eqn-W10} and the maximum principle to conclude that 
$W \leq 0$ for all $t \in [0,T^*]$.  
To this end,  observe first that $u, \bu \geq 1$  implies that for every  fixed $\e >0$ and for all $t \in (0,T)$, 
\be\label{eqn-WWW} 
m^*:=\sup_{(x,t) \in \R^n \times (0,T^*]} W(x,t) \leq  \frac 1{\e^2}.
\ee
Indeed, notice that if there is a point  $(x,t) \in \R^n \times (0,T^*]$ where $W(x,t) \geq 0$, then since $\bu \geq 1$, at such point we have
$u \geq \bu +  \e (t + \e) \, u^2 \geq \e^2 \, u^2$, that is $u(x,t) \leq \e^{-2}$. Hence, $W(x,t) \leq u(x,t) \leq \e^{-2}$ and
the same holds  for the supremum $m^*$.

\smallskip

\begin{claim}\label{claim-1} 
We have 
\bee
m^*:=\sup_{(x,t) \in \R^n \times (0,T^*]} W(x,t)  \leq  0
\eee
provided that $\e$ is sufficiently small. 

\end{claim}

Once this claim is shown, the theorem will follow by simply letting $\e \to 0$ to show that $u \leq \bu$ and then switching the roles of $u$ and $\bu$. 
%%The proof   above shows that {\em there is no first time $t_0 >0$ at which $W^*(t_0) := \sup_{x \in R^n} W(x,t_0) = 0$. } 
%
%\smallskip 
%Since we are dealing with complete manifolds and the continuity of $W^*(t)$
%at $t=0$ is not guaranteed, to prove Claim \ref{claim-1} first needs to show that   $ \limsup_{t \to 0} W^*(t) <0$, as stated in the next claim. 
%
%\begin{claim}\label{claim-2} 
%We have 
%\bee
%\limsup_{t \to 0} W^*(t) \leq  - \frac \e2 <0.  
%\eee
%\end{claim} 
%
%
%Assuming that Claim \ref{claim-2} holds, we will prove next Claim \ref{claim-1}. Then, we will prove Claim \ref{claim-2}. 

%we still need to  rule out the case that $W^*(t) >0$,
%for all $t >0$, that is $W(x,t)$ becomes instantly positive at spatial infinity. } We will do that next. In such case there is a sequence $t_k \to 0$ such that $W^*(t_k) > 0$.
%Then, either $W^*(t_k) = W(x_k, t_k)$ for some $x_k \to +\infty$ or $W^*(t_k)$. Let's first deal with the first case. 

\begin{proof}[Proof of Claim \ref{claim-1}] 

To prove the claim, we assume by contradiction, that  
$$m^* >0.$$
Since $\lim_{|x| \to +\infty} u(x,t) = +\infty$ uniformly in $[0,T]$ and $\bu \geq 1$, the supremum $m^*$ cannot be attained at infinity. Hence, we have 
$$m^* = W(x_{\max}(t_0), t_0)$$ for some point $t_0 \in (0, T^*]$ and $x_{\max}(t_0) \in \R^n$. Then at such point  
\be\label{eqn-uuu}
 \big (1 - \e (t_0+\e)  u \big ) \, u  = \bu + m^* \qquad \mbox{and} \qquad \big (1-2\e (t_0+\e)  u \big ) D_i u = D_i \bu
 \ee
Note that the first equality, $m^* >0$  and $ \bu \geq 1$ imply that $1- \e (t_0+\e)  u >0$ at the maximum point, which will be used below. 
We will now use the second equality in \eqref{eqn-uuu} to evaluate the  right hand side of \eqref{eqn-W} at the maximum point. First, we have 
\begin{equation}\label{eqn-aaa}
\begin{split}
a_{ij} - \ba_{ij} &= \frac{D_i \bu D_j \bu}{1+|D\bu|^2} - \frac{D_i u D_j u}{1+|Du|^2} =
(1-  2\e (t_0+\e) u)^2 \,   \frac{D_i u D_j u}{1+|D\bu|^2} - \, \frac{D_i u  D_j u }{1+|Du|^2}  \\
&=  \frac{D_i u D_j u}{(1+|Du|^2)(1+|D\bu|^2)} \left [ (1- 2\e \, (t_0+\e)\,u )^2\, (1+|Du|^2)-  (1+ |D\bu|^2) \right ]\\
&=- 4\e (t_0+\e) \,u  \big (1- \e (t_0+\e)\, u \big )   \, \frac{D_i u D_j u}{(1+|Du|^2)(1+|D\bu|^2)}.
\end{split}
\end{equation}
To derive the last equality we  used $(1- 2\e (t_0+\e)\,u )^2\,|Du|^2 = |D\bu|^2$ which gave us
$$(1- 2\e (t_0+\e) u )^2\, (1+|Du|^2)-  (1+ |D\bu|^2) = (1- 2\e (t_0+\e) u )^2 -1  =  - 4\e (t_0+\e) u \,  \big (1- \e (t_0+\e) u
\big ).$$
Combining the above with \eqref{eqn-W} we find that at the  point $(x_{\max}(t_0),t_0)$  we have 
\be\label{eqn-W2} 
0 \leq W_t - a_{ij} D_{ij} W <   -  4\e (t_0+\e) u\, \big (1- \e (t_0+\e)\, u \big )  \, \frac{D_{ij} \bu D_i u  D_j u}{(1+|Du|^2)(1+|D\bu|^2)}  -  \frac \e 2 u^2. 
\ee
%\smallskip 
%
%We next distinguish the following two cases: 
%
%\noindent{\em  Case 1:} $1- 2\e (t_0+\e)\,u =0$.  In this case we have $D\bu =0$. Recalling that  
% $\bar h^i_j=  \frac{D_{ij} \bu}{\sqrt{1+|D\bu|^2}} -\frac{D_{lj} \bu D_l \bu D_i \bu }{(1+|D\bu|^2)^{3/2}}$ we get  $\bar h^i_j =  D_{ij} \bu$. Hence  \eqref{eqn-cond-curv} applied to $\bar u$ gives 
%$D_{ij} \bu \geq -c \delta_{ij}$. Inserting this bound in  \eqref{eqn-W2} and using also that $D\bu =0$ we obtain that at the maximum point $(x_0, t_0)$ we have
%\be\label{eqn-W25}
%0 \leq W_t - a_{ij} D_{ij} W <    4 c \e (t_0+\e) u\, \big (1- \e (t_0+\e)\, u \big )  -  \frac \e 2 u^2. 
%\ee
%
%
%\medskip 
%
%\noindent{\em Case 2:} $1- 2\e (t_0+\e)\,u \neq 0$.  

We next use the lower bound  on the second fundamental form in \eqref{eqn-cond-curv} which implies that
$$ \bar{v} \bar{h}^i_j \, D_i u D_j u \geq - c\, |Du|^2.$$
On the other hand, since  $\bar h^i_j=  \frac{D_{ij} \bu}{\sqrt{1+|D\bu|^2}} -\frac{D_{lj} \bu D_l \bu D_i \bu }{(1+|D\bu|^2)^{3/2}}$, it follows that  
at the maximum point $(x_{\max}(t_0), t_0)$ we have 
\begin{align*} \bar{v} \bar{h}^i_j \, D_i u D_j u&=  \Big (  D_{ij} \bu  -\frac{D_{lj} \bu D_l \bu D_i \bu }{1+|D\bu|^2} \Big )  D_i u D_j u\\
&=  D_{ij}\bu \,D_i u D_j u- \langle D\bu, Du\rangle \frac{D_{ij}\bu}{1+|D\bu|^2} \,  D_i \bu D_j u \\
&=\left(1+|D\bu|^2-(1- 2\e \, (t_0+\e)\,u )^2 |Du|^2\right)\frac{D_{ij} \bu}{1+|D\bu|^2}  D_i u D_j u\\
&=\frac{D_{ij} \bu D_i u D_j u}{1+|D\bu|^2}.
\end{align*}
Combining the last two formula gives  
%
%
%
%
%\bee
% \begin{split}
% \frac{\bar h^i_j  \, D_i \bu D_j \bu}{\sqrt{1+|D\bu|^2} } &= \frac {D_{ij} \bu D_i \bu D_j \bu } {1+|D\bu|^2 } -   
% \frac{ (D_i \bu)^2  D_{lj} \bu D_l \bu D_j \bu} {(1+|D\bu|^2)^2}\\
%&= \frac{D_{ij} \bu D_i \bu D_j \bu } {1+|D\bu|^2} -  \frac{ (D_l \bu)^2  D_{ij} \bu D_i \bu D_j \bu}{(1+|D\bu|^2)^2}\\& = \frac{D_{ij} \bu D_i \bu D_j \bu } {(1+|D\bu|^2)^2}.
%\end{split}
% \eee
%Hence the lower bound  \eqref{eqn-cond-curv} applied to $\bar u$, $\bar v=\sqrt{1+|D\bu|^2}$ and the second fundamental form $\bar h_i^j$  yields  
\be\label{eqn-lowerb} \frac{ D_{ij} \bu D_i u D_j u}{1+|D\bu|^2}  =  \bar h^i_j  \, \bar v \, D_i u D_j u \geq - c\, |Du |^2. 
\ee
Inserting this bound in \eqref{eqn-W2},  implies   that at the  point $(x_{\max}(t_0),t_0)$  we have 
\be\label{eqn-W5}
\begin{split}
0 \leq W_t - a_{ij} D_{ij} W &<  4\e c  \, (t_0 +\e)   u   (1- \e (t_0+\e)\, u)  \frac{|Du|^2}{ 1+|Du|^2} -  \frac \e 2 u^2\\ 
&\leq  4\e c  \, (t_0 +\e)  u  (1- \e (t_0+\e)\, u)  -  \frac \e 2 \, u^2. 
\end{split} 
\ee
We conclude from  \eqref{eqn-W5} that at the maximum point  $(x_{\max}(t_0),t_0)$ we have 
$$4\e c\,  ( t_0 +\e) u   (1- \e (t_0+\e)\, u)  -  \frac \e 2 u^2 >0$$
holds at the maximum point  $(x_{\max}(t_0),t_0)$, that is
$$u < 8 c t_0 \,  (1- \e (t_0+\e)\, u) < 8 c \, (t_0 +\e)$$
holds, since $1- \e (t_0+\e)\, u>0$. 
Then $u \geq 1$ yields  that $t_0 + \e >  \frac 1{8c}$, where $c$ is the constant from \eqref{eqn-cond-curv}. 
Since we have assumed that $t_0 \in (0, T^*]$ and $T^* \leq \frac 1{10c}$ we derive a contradiction by choosing $\e$ sufficiently small. This shows, that 
contrary to our assumption, $W^*(t_0) < 0$, finishing the proof of the claim. 
\end{proof} 

\smallskip

We have just seen that $W:= u - \bu - \e ( t+\e) u^2  \leq 0$ on $\R^n \times (0, T^*]$. Let $\e \to 0$
to obtain that $u \leq \bu$ on $\R^n \times (0, T^*]$. Similarly, $\bu \leq u$ on the same interval, which means that
$u=\bu$. By repeating the same proof starting at $t=T^*$ we conclude after finite many steps that
$u \equiv \bu$ on $\R^n \times (0,T)$, finishing the proof of the theorem. 

\end{proof}

\subsection{Proof of Theorem \ref{thm-general-sub}.} 

\begin{proof}
The proof of Theorem \ref{thm-general-sub}  is very similar to that of  Theorem \ref{thm-general}.  We briefly outline it in what follows. 
As before, let  $u, \bu: \cD:= \cup_{t\in(0,T]} \big ( \Omega_t\times\{t\} \big )  \to \mathbb{R}$ be the two smooth solutions to  \eqref{eqn-IVP-B} with  initial data $u_0$ as in the statement of Theorem \ref{thm-general-sub};
(as above, we simplify the notation by calling $u=u_1$ and $\bu = u_2$).  
%We assume in addition that $u(x,t) \to +\infty$ and  $\bu(x,t) \to +\infty$ as $|x| \to +\infty$, for all $t \in (0,T)$. 
Our  assumption that $u_0$ is proper implies that   $u_0\geq -C$  for some constant $C >0$  and hence Lemma \ref{lemma-llb} implies that $u, \bu  \geq -C$, for $t >0$ (possibly for a different constant $C>0$ which is    uniform in $t$ for  
$t < \min (1, T)$, where $T$ is the maximal existence time).  
 By adding on both solutions the constant $C+1$ we may assume that 
$u, \bu  \geq 1$.
As in the proof of Theorem \ref{thm-general}, we take 
$$W := w - \zeta - \e=u-\bu - \e \, (t +\e)  \, u^2.$$

Let  $m^*:=\sup_{(x,t) \in \cD} W(x,t)$ and assume that $m^*>0$.\\

We first remark that Lemma \ref{lemma-llb} and Corollary \ref{cor-proper} can be directly extended to estimate the infimum of $u$ in  $\cD\cap  B_R(x_0)$ (instead of  $\mathbb{R}^n\cap  B_R(x_0)$ ).  Hence we have that if $u_0$ is proper then  $u(x, t)\to \infty$ uniformly in $t$ as $|x|\to \infty$.\\

 Let 
$(x_k, t_k)$ be a sequence of points in $\cD$ such that  $W(x_k, t_k)\to m^*$. Note that from our definition and the previous remark we have that if $t_k \to \bar{t}$,   and  either $x_k\to \partial \Omega_{\bar{t}}$ or $|x_k| \to +\infty$, then $u(x_k, t_k)\to \infty$ and $W\to -\infty$. Hence, we may assume that that supremum of $W$ is attained in the interior of $\Omega_{\bar{t}}$.
Now we conclude the desired result by following the  the proof of Theorem \ref{thm-general}. 
\end{proof}

\subsection {Extension of uniqueness for  entire  graphs (not necessarily proper).} 
In this section we provide extensions to our result in Theorem \ref{thm-general}.   We will consider graphical solutions that are not necessarily proper, but their initial height function $u_0$ and its gradient  function $v_0$ satisfy the following assumption 
 \begin{equation}\textrm{ for every $M$ there is a constant $c(M)$ such that }\sup_{\{x: u_0(x))<M\}} v_0\leq c(M). \label{eq:boundedcase}\end{equation}

This condition can be understood as excluding oscillatory behavior in the set where the height function  $u_0$ is bounded at the initial time. Then our result states as follows: 

\begin{thm}\label{thm-general-np} Assume that $u_0: \R^n \to \R$ is a locally Lipschitz function $($not necessarily proper$)$  defining an entire graph hypersurface $M_0=
\{ (x, u_0(x)): \, x \in \R^n \} \subset \R^{n+1}$ whose height function $u_0$ is bounded from below  and also satisfies   
condition \eqref{eq:boundedcase}. 

Let $u_1, u_2 \colon \R^n \times (0,T]\to \R$ be two  smooth  solutions of \eqref{eqn-IVP} defining two entire graph solutions $M_t^1 =
\{ (x, u_1(x, t)): \, x \in \R^n \} $ and $M_t^2 =\{ (x, u_2(x, t)): \, x \in \R^n \} $ of MCF \eqref{eqn-MCF} satisfying condition  \eqref{eqn-cond-curv} and having the  same initial data $u_0$, that is  $\lim_{t \to 0} u_1(\cdot, t) =  \lim_{t \to 0} u_2(\cdot, t) =u_0$.  Then, $u_1=u_2$ on $\R^n \times (0,T]$,
that is $M_t^1 = M_t^2$ for all $t \in (0,T]$. 
\end{thm}

We will first show that condition \eqref{eq:boundedcase} is preserved in time and that implies uniform local bounds for the second fundamental form on the set where $\{ u \leq M \}$ (these bounds depend only on  $M$). 

\begin{prop} \label{second fund form bound} 
Assume that $u\geq 0$ is a smooth solution of \eqref{eqn-IVP} with initial data $u_0$ and that \eqref{eq:boundedcase} holds. Then,
\begin{enumerate}[i$)$]
\item\label{vbound} $(M-u)_+^2 \, v\leq M^2\,c(M)$ holds for all $t \in (0, T]$. 

\item If, we further assume that   $|A|^2(x,0)\leq c(M)$ in the set  $\{x\, :\, u_0(x)\leq M\}$ $($without loss of generality we can take $c(M)$ to be the same as in \eqref{eq:boundedcase}$)$,   then
\begin{equation} |A|^2 \,  (M-u)^2_+ \leq \max\{ c(M)\, M^2, \, k^{-1} (3+k^{-1} ) M\} \label{interior max}\end{equation} 
when $u(x,t)\leq M$  and $k=\frac{1}{2\,M^2 \,c(M)}$.

\item Without any assumption  on the second fundamental form at the initial time, we  have instead 
\begin{equation} t |A|^2  \,(M-u)^2 (x,t) \leq  2k^{-1} (3+k^{-1} ) M+ M^2  \label{interior max in t}\end{equation} 
if $u(x,t)\leq M$  and $k=\frac{1}{2\,M^2 \,c(M)}$.

\end{enumerate}
\end{prop}
\begin{proof}
\begin{enumerate}[i)]

\item Consider the cut-off function (in terms of both $u$ and $\x$) given by 
\begin{equation}\eta_R(x,t)=\left((M-u)_+\left(1-\frac{|\x|^2+2nt}{R^2}\right)_+-\frac{4}{R}t\right)_+. \label{def-eta}\end{equation}
A direct calculation shows that 
\be\label{eqn-hhh}
(\eta_R)_t-\Delta \eta_R=\frac{2}{R^2}\langle \nabla u, \nabla |\x|^2\rangle-\frac{4}{R}\leq 0.
\ee
In the last line we used that $|\nabla |\x|^2|=2 |x^T|\leq 2 R$ in the set that $1-\frac{|\x|^2+2nt}{R^2}\geq 0$ and that $|\nabla u|\leq 1$. Recalling also   that
$$v_t-\Delta v=-|A|^2v-2\frac{|\nabla v|^2}{v}$$
and defining $V_R=v\, \eta_R^2$ we have \begin{align*} (V_R)_t-\Delta V_R &= \eta_R^2\left( -|A|^2v-2\frac{|\nabla v|^2}{v}\right)-2v|\nabla \eta_R|^2-4\eta \langle \nabla v, \nabla \eta\rangle\\
&\leq  \eta_R^2\left( -|A|^2v-2\frac{|\nabla v|^2}{v}\right)-2v|\nabla \eta_R|^2+ 2\eta_R^2 \frac{|\nabla v|^2}{v}+2|\nabla \eta_R|^2 v\\
&=- \eta_R^2 |A|^2v<0.
\end{align*}
A standard application of the maximum principle shows that  $V_R$ does not have any interior maximum and hence $$V_R\leq \max V_R(\cdot, 0)\leq M^2 c(M).$$ The result follows by taking $R\to \infty$.

\item We follow the proof in \cite{EH2} replacing the localization function in that paper by $\eta_R^2$ (where $\eta_R$ is defined by \eqref{def-eta}).  The proof is analogous and we only point out the main steps and differences. 

Following \cite{EH2} we define $k$ such that 
$kv^2\leq \frac{1}{2}$ in the set that $\eta_R\not=0$ and define the function 
$$g= \frac{v^2\, |A|^2}{1-kv^2}.$$
Then 
$$g_t-\Delta g\leq -2k g^2-\frac{2k}{(1-kv^2)^2}|\nabla v|^2 g-2 \frac{v^{-1}}{1-kv^2}  \langle \nabla v, \nabla g\rangle.$$

A similar calculation as in \cite{EH2} where we use \eqref{eqn-hhh} gives that 
\begin{align*}
(\eta^2_R\,g)_t-\Delta (\eta^2_R\,g)&\leq -2k \eta^2_R g^2-\frac{2k}{(1-kv^2)^2}|\nabla v|^2  \eta^2_R g-2 \eta^2_R \frac{v^{-1}}{1-kv^2}  \langle \nabla v, \nabla g\rangle\\
&-
2g |\nabla \eta_R|^2-4 \eta_R \langle \nabla \eta_R , \nabla g\rangle.
\end{align*}

Following again \cite{EH2} we can find a vector function $b$ (that can be explicitly computed, but it is not important) such that
%\begin{align*}
%-2 \eta^2_R \frac{v^{-1}}{1-kv^2}  \langle \nabla v, \nabla g\rangle & = -2  \frac{v^{-1}}{1-kv^2}  \langle \nabla v, \nabla (g \, \eta^2_R)\rangle +4  g \eta_R
%\frac{v^{-1}}{1-kv^2}  \langle \nabla v, \nabla \, \eta_R\rangle 
%\\&\leq -2  \frac{v^2}{1-kv^2} v^{-3} \langle \nabla v, \nabla (g \, \eta^2_R)\rangle+\frac{2k}{(1-kv^2)^2}|\nabla v|^2  \eta^2_R g+ 2k^{-1} v^{-2} g |\nabla \eta_R|^2
%\end{align*}

%We also have
%\begin{align*}
%-4 \eta_R \langle \nabla \eta_R , \nabla g\rangle=-4  \eta_R^{-1} \langle \nabla \eta_R , \nabla  (g \eta_R^2)\rangle+8 g |\nabla \eta_R|^2
%\end{align*}
\begin{align*}
(\eta^2_R\,g)_t-\Delta (\eta^2_R\,g)&\leq -2k \eta^2_R g^2+(6+2k^{-1} v^{-2}) g |\nabla \eta_R|^2+ \langle \nabla  (g \eta_R^2), b \rangle.
\end{align*}
Then, observing that  $|\nabla \eta_R|^2\leq M$ we conclude that if $\eta^2_R\,g$ has an interior maximum then 
 $$0\leq -2k \eta^2_R g^2+(6+2k^{-1} v^{-2}) g |\nabla \eta_R|^2\leq -2k \eta^2_R g^2+(6+2k^{-1} v^{-2}) g M $$
or equivalently, 
$$ \eta^2_R g\leq k^{-1} (3+k^{-1} v^{-2}) M.$$ Taking $R$ to infinity  \eqref{interior max} follows since $v\geq 1$.

\smallskip 
\item Finally,   consider $ t \, \eta^2_R\,g$. Then,  we have 
$$(t\eta^2_R\,g)_t-\Delta (t\eta^2_R\,g)\leq -2k \eta^2_R g^2+(6+2k^{-1} v^{-2}) g |\nabla \eta_R|^2+ \langle \nabla  (g \eta_R^2), b \rangle+ \eta^2_R\,g .$$
At a maximum holds
$$ t\, \eta^2_R \, g\leq k^{-1} (3+k^{-1} v^{-2}) M+ M^2,$$ and we conclude  \eqref{interior max in t} by taking $R\to \infty$.

\end{enumerate}

\end{proof}

We will now  prove Theorem \ref{thm-general-np}:

\begin{proof}[ Proof of Theorem \ref{thm-general-np}] 

As in the proof of Theorem \ref{thm-general}, we set $u=u_1$, $\bu=u_2$ and assume without loss of generality that $u_0 \geq 1$ in which case $u, \bu \geq 1$ (this follows from $u_0 \geq 1$ 
and Lemma \ref{lemma-llb}). We define as before  
$$W := w - \zeta=u-\bu - \e \, (t +\e)  \, u^2$$
and  set 
$$T^*= \min \Big (  T, \frac 14, \frac 1{10\bar c} \Big )$$
where $\bar c$ is a uniform constant (to be determined later) and depends on the  constant  $c$ in \eqref{eqn-cond-curv}. 

We   proceed as in the proof of Theorem \ref{thm-general}, but we need to consider an additional case:  {\em the supremum $m^*$ is attained at infinity. }   This means,  
 there exists a  sequence of points $y_k  \in \R^n$ with $|y_k| \to +\infty$ 
and a sequence of times $s_k \in (0, T^*]$, $s_k \to  t_0$  such that 
$$W(y_k, s_k) > \frac{m^*}2 >0.$$

Applying the maximum principle 
we will deduce that $t_0 > 1/8c$ deriving a contradiction to the definition of $T^*$. 
Notice that since our initial data is complete non-compact and the convergence of our solutions to the initial  data is assumed to be uniform only on compact subsets 
of $\R^n$, it is not a'priori guaranteed that  $t_0 >0$, that is at this point we assume that $s_k \to t_0 \in [0,T^*]$.  

\smallskip  

To apply the maximum principle, we employ a parabolic version of the  Omori-Yau maximum principle (see for example in  \cite{Ma}). We define the functions  
$$W_k(x,t) =W(x,t) -t\frac{|x|^2}{C_k^2}, \qquad \mbox{for }\,\,C_k=\max\{|y_k|^2, k\} $$
and we look  at the supremum of $W_k$ in $\R^n \times (0,s_k]$. If this supremum is less than  ${m^*}/4$,  then   $W(y_k,s_k)
\leq \frac{m^*}4  + t \frac{|y_k|^2}{C_k^2}$ and from our  choice of $C_k$ we have $W(y_k,s_k)  \leq  \frac{3m^*}8 < \frac{m^*}2$ for $k \gg 1$, 
contradicting our assumption. 

We deduce  that $m_k:= \sup_{\R^n \times (0, s_k]} W_k > \frac{m^*}4 >0$. Since  $W$ is uniformly bounded (see \eqref{eqn-WWW}) 
this supremum is attained  in the interior at a point $(x_k, t_k) \in \R^n \times (0, s_k]$.  At this point necessarily we have  
\be\label{eqn-omori0}
\begin{split}
&W(x_k,t_k) \geq t_k \,  \frac{|x_k|^2}{C_k^2} >0, \qquad W_t(x_k, t_k)= (W_k)_t(x_k, t_k)+   \frac{|x_k|^2}{C_k^2}\geq 0\\
&D W(x_k,t_k) = \frac{2\, t_k\, x_k}{C_k^2}, \qquad \quad \,\, D_{ij} W(x_k,t_k) \leq  \frac{2 t_k\delta_{ij}}{C_k^2} \leq  \frac{2 t_k\delta_{ij}}{k^2}\end{split}
\ee
where the last inequality is understood in the sense of quadratic forms,  that is  for all $\xi \in \R^n \setminus \{0\}$, $D_{ij} W(x_k,t) \xi_i \xi_j < \frac{2t_k}{k^2} \, |\xi|^2$ holds. Furthermore, notice that since $(x_k, t_k)$ is the maximum for $W_k$ on $\R^n \times (0,s_k]$, we have $W(x_k, t_k)-t_k\frac{|x_k|^2}{C_k^2}\geq W(0,0)$, and because $W \leq \e^{-2}$, we have 
 $$\frac{t_k\,|x_k|^2}{C_k^2}\leq  W(x_k, t_k) - W(0,0) \leq \e^{-2} - W(0,0) = \e^{-2} + \e^2 \, u^2 (0,0) =: M_\e.$$
%Now we observe that $|x_k|\to \infty$ as $k\to \infty$, since otherwise (up to subsequence) $x_k\to \bar{x}$. Then it would hold
%$$W(y_k, t_k)- t_k \frac{|y_k|^2}{C_k^2} \leq W(x_k, t_k)-t_k  \frac{|x_k|^2}{C_k^2} \to W(\bar{x}, t_0).$$
%From our choice of $C_k$ we would have $m^*\leq W(\bar{x}, t_0)$, contradicting that the supremum is attained (only) at infinity. In particular,  we may assume that $|x_k|\geq 1$ and we get
Then 
\be\label{eqn-Wkkk}
 |D W(x_k,t_k) | = \frac{2 t_k\, |x_k|}{C_k^2} \leq \frac{2  \sqrt{ t_k\, M_\e} }{C_k }\leq \frac{2  \sqrt{ t_k\, M_\e} }{k } = \cO( \frac {\sqrt{t_k}}k).
 \ee
Moreover, since $W_k(x_k, t_k) = m^*_k > \frac{m^*}4 >0$  we have  
$W(x_k,t_k) = W_k(x_k, t_k) + t_k \frac{|x_k|^2}{C_k^2} >\frac{m^*}4 >0.$
Combining these with  \eqref{eqn-omori0} we conclude  the following:
\be\label{eqn-omori}
W(x_k, t_k) > \frac{m^*}4 >0, \quad W_t(x_k, t_k) \geq 0, \quad   |D W(x_k,t_k) | \leq \cO( \frac{\sqrt{t_k}}k),   \quad  D_{ij} W(x_k,t_k)  \leq \frac{2 \delta_{ij}}{k^2}.  
\ee

Hence, we deduce from  \eqref{eqn-w},  \eqref{eqn-zeta},  \eqref{eqn-omori} and the uniform ellipticity of the matrix $a_{ij}$,  that 
\be\label{eqn-Wk} 
- \frac Ck \leq W_t - a_{ij} D_{ij} W <   ( a_{ij} - \ba_{ij}  )\, D_{ij} \bu -  \frac \e 2 u^2 
\ee
 holds    at each point $(x_k, t_k)$. 
Furthermore from $W(x_k,t_k)>0$ we have 
\be\label{eqn-uk}
 (1 - \e (t_k +\e)  u) \, u(x_k,t_k)  + \e  >   \bu(x_k,t_k).  
 \ee

Next,  observe  that the fact that $W(x_k,t_k)>0$ 
 implies that $u(x_k,t_k) $ is bounded (otherwise if  $u(x_{k_l},t_{k_l}) \to +\infty$ for some subsequence, then  $\lim_{l \to +\infty} W(x_{k_l}, t_{k_l})  \to -\infty$). Furthermore, $u(x_k,t_k) $  bounded  and $u, \bu \geq 1$ imply  that $\bu (x_k,t_k)$ is bounded as well. Hence, we may assume without loss of generality that
\be\label{eqn-ubu}
 u(x_k,t_k) \to u^* \qquad \bu(x_k,t_k) \to \bu^* \qquad \mbox{and} \qquad  1 \leq u(x_k,t_k),  \bu(x_k,t_k) \leq u^*+1.
 \ee
Therefore,  our assumption that $u, \bu$ satisfy 
condition \eqref{eq:boundedcase} and the first assertion in Proposition \ref{second fund form bound} applied to $M=u^*+2$  yield
\be\label{eqn-ubDu}
| D u(x_k,t_k) | \leq C( u^*)  \qquad \mbox{and} \qquad  | D \bu(x_k,t_k) | \leq C( u^*).
 \ee
Furthermore, by the third assertion in Proposition \ref{second fund form bound} we have 
\bee\label{eqn-Ab}
t_k \, |A|^2 (x_k, t_k)  \leq C(u^*)  \qquad \mbox{and} \qquad  t_k \, |\bar A|^2 (x_k, t_k) \leq C( u^*).
 \eee
It follows that at the points $(x_k,t_k)$ we have 
\be\label{eqn-hb}
\sqrt{t_k}  \, v |h_i^j|  (x_k, t_k)  \leq C(u^*)  \qquad \mbox{and} \qquad \sqrt{t_k}  \, v |\bar h_i^j | (x_k, t_k) \leq C(u^*)
 \ee
 and also 
\be\label{eqn-ddb}
\sqrt{t_k} \frac{|D_{ij} u|}{\sqrt{1+|Du|^2}}   \leq C(u^*)  \qquad \mbox{and} \qquad \sqrt{t_k} \frac{|D_{ij} \bu|}{\sqrt{1+|D\bu|^2}}  \leq C(u^*). 
 \ee 
These bounds will be used momentarily.

We will next analyze the main term on right hand side of \eqref{eqn-Wk}.
From the definition of $W$ we have that $D\bu(x_k, t_k)=(1- 2\e \, (t_k+\e)\,u )Du - DW$. Then, similarly to \eqref{eqn-aaa}
(the computation here has more terms since $DW \neq 0$) we get 
\begin{equation*}\label{eqn-aak}
\begin{split}
a_{ij} - \ba_{ij} &=  \frac{D_i \bu D_j \bu}{1+|D\bu|^2} - \frac{D_i u D_j u}{1+|Du|^2} = (1-  2\e (t_0+\e) u)^2 \,   \frac{D_i u D_j u}{1+|D\bu|^2} - \, \frac{D_i u  D_j u }{1+|Du|^2}\\
&+ \frac{D_i W D_j W -(1- 2\e \, (t_k+\e)\,u ) (D_i u D_j W+D_i W D_j u)}{1+|D\bu|^2} \\
&=\big(-4\e (t_0+\e) \,u  \big (1- \e (t_0+\e)\, u \big )  + \langle DW, b\rangle\big)  \, \frac{D_i u D_j u}{(1+|Du|^2)(1+|D\bu|^2)}\\&\, + \frac{D_i W D_j W -(1- 2\e \, (t_k+\e)\,u ) (D_i u D_j W+D_i W D_j u)}{1+|D\bu|^2}
\end{split}
\end{equation*}
where $b=2 (1- 2\e \, (t_k+\e)\,u ) Du  -DW$.
Denoting  $$B_{ij}=  \langle DW, b\rangle  \, \frac{D_i u D_j u}{(1+|Du|^2)(1+|D\bu|^2)}+ \frac{D_i W D_j W -(1- 2\e \, (t_k+\e)\,u ) (D_i u D_j W+D_i W D_j u)}{1+|D\bu|^2}$$
we can then express the main term  $(a_{ij} - \ba_{ij}) \, D_{ij} \bu$ on right hand side of \eqref{eqn-Wk} as 
\be\label{eqn-aak}
(a_{ij} - \ba_{ij}) \, D_{ij} \bu = -4\e (t_0+\e) \,u  \big (1- \e (t_0+\e)\, u \big )  \, \frac{D_{ij } \bu D_i u D_j u}{(1+|Du|^2)(1+|D\bu|^2)} + B_{ij} \, D_{ij} \bu. 
\ee

Next  observe that from \eqref{eqn-Wkkk} at $(x_k, t_k) $ we have that 
$$|B_{ij}|\leq C(u^*) \frac{\sqrt{t_k}}{k}$$
which combined  with \eqref{eqn-ddb} yields 
\be \label{eqn-term2}
|B_{ij} D_{ij} \bu |\leq  C \frac{\sqrt{t_k}}{k} (\sqrt{t_k})^{-\frac{1}{2}} = \cO(\frac 1k). 
\ee

To bound the first term on the right-hand side of \eqref{eqn-aak} we use  \eqref{eqn-cond-curv} which in particular implies that 
\be\label{eqn-hhhh} 
\bar{v} \bar{h}^i_j D_i u D_j u\geq -c\, |Du|^2.
\ee
On the other hand,  $\bar h^i_j=  \frac{D_{ij} \bu}{\sqrt{1+|D\bu|^2}} -\frac{D_{lj} \bu D_l \bu D_i \bu }{(1+|D\bu|^2)^{3/2}}$ implies 
\begin{align*} \bar{v} \bar{h}^i_j \, D_i u D_j u&=  \Big (  D_{ij} \bu  -\frac{D_{lj} \bu D_l \bu D_i \bu }{1+|D\bu|^2} \Big )  D_i u D_j u\\
&=  D_{ij}\bu \,D_i u D_j u- \langle D\bu, Du\rangle \frac{D_{ij}\bu}{1+|D\bu|^2} \,  D_i \bu D_j u \\
&=\left(1+|D\bu|^2-(1- 2\e \, (t_k+\e)\,u )^2 |Du|^2\right)\frac{D_{ij} \bu}{1+|D\bu|^2}  D_i u D_j u \\ &\,\quad -(1- 2\e \, (t_k+\e)\,u )\frac{|Du|^2}{1+|D\bu|^2} D_{ij} u  D_j u D_i W
- \frac{\langle DW, Du\rangle}{1+|D\bu|^2} D_{ij} u \, D_i \bu D_j u\\
&=\left(1+|D\bu|^2-(1- 2\e \, (t_k+\e)\,u )^2 |Du|^2\right)\frac{D_{ij} \bu}{1+|D\bu|^2}  D_i u D_j u + \cO(\frac 1k)
\end{align*}
where to derive the last line we combined \eqref{eqn-Wkkk} and \eqref{eqn-ddb} (following a similar estimate as the one we did for $B_{ij} D_{ij}\bu$). 

To further estimate the last line above, we use  $$|D\bu|^2-(1- 2\e \, (t_k+\e)\,u )^2 |Du|^2=\langle DW, D\bu+ (1- 2\e \, (t_k+\e)\,u )\, Du\rangle =\cO\big( \frac{\sqrt{t_k}}{k})$$
concluding that 
$$\bar{v} \bar{h}^i_j \, D_i u D_j u =  \frac{D_{ij} \bu}{1+|D\bu|^2}  D_i u D_j u + \cO(\frac 1k)$$
which in turn, combined with \eqref{eqn-hhhh} yields 
\be\label{eqn-bububu}
\frac{D_{ij} \bu}{1+|D\bu|^2}  D_i u D_j u\geq -c\,  |Du|^2 +\cO\big(\frac{1}{k}\big).
\ee

Finally,   \eqref{eqn-Wk}, \eqref{eqn-aak}, \eqref{eqn-term2} and \eqref{eqn-bububu}  together imply  that as $k\to \infty$
$$0 \leq    4 \bar c \, \e (t_0+\e) \,u^*  \big (1- \e (t_0 +\e)\, u^* \big )   - \frac \e2 \, (u^*)^2.$$
We now use the same argument as in the proof of Theorem \ref{thm-general} to conclude that this is not possible provided that  $t_0 + \e >  \frac 1{8c}$, where $c$ is the constant from \eqref{eqn-cond-curv}. 
Since we have assumed that $t_0 \in (0, T^*]$ and $T^* \leq \frac 1{10 \bar c}$ we derive a contradiction by choosing $\e$ sufficiently small. This shows, that 
contrary to our assumption, $W^*(t_0) < 0$, finishing the proof of the claim.

\end{proof}

\section{The convex case and Harnack inequality}\label{sec-convex} 

In this final section we will state and existence and  uniqueness result for convex, proper,  non-compact  entire graphs  Mean curvature flow solutions
and show that Hamilton's Harnack inequality holds. 

\begin{thm}[Uniqueness  of convex entire graph solutions]
 \label{thm-general-convex} Assume that $u_0: \R^n \to \R$ is a convex  function defining a proper  entire graph 
convex hypersurface $M_0=
\{ (x, u_0(x)): \, x \in \R^n \} \subset \R^{n+1}$.  Let $u_1, u_2 \colon \R^n \times (0,T) \to \R$ be two solutions of \eqref{eqn-IVP} defining two proper smooth convex entire graph solutions $M_t^1 =
\{ (x, u_1(x, t)): \, x \in \R^n \} $ and $M_t^2 =\{ (x, u_2(x, t)): \, x \in \R^n \} $ of MCF \eqref{eqn-MCF} with the same initial data $u_0$, that is  $\lim_{t \to 0} u_1(\cdot, t) =  \lim_{t \to 0} u_2(\cdot, t) =u_0$. Then,  $u_1=u_2$ on $\R^n \times (0,T)$,
that is $M_t^1 = M_t^2$ for all $t \in (0,T)$. 
\end{thm}  

\begin{proof}  Now, since our initial data is a  convex proper entire graph
over $\R^n$, we may assume that it lies above the $e_{n+1}=0$ plane, that is $u_0(x) \geq 0$ for all $x \in \R^n$. Furthermore, we have $\lim_{x \to +\infty} u_0(x) = +\infty$
and the same holds for both solutions $u_i(x,t)$, $i=1,2$, namely  $u_i(\cdot,t) \geq 0$ and  $\lim_{x \to +\infty} u_i(x,t) = +\infty$,  for all $t >0$. 

\smallskip 

One then can then  apply the maximum principle argument in  Theorem \ref{thm-general} (actually in the convex case the computation is simpler) to show that for any small number $\epsilon >0$,  one  has $u_1 - u_2  \leq \e t  \, u_1^2 + \e$ and,  similarly,   $u_2 - u_1  \leq \e t u_2^2 + \e$, for all $t \in (0,T)$. Taking $\e \to 0$ readily gives that $u_1=u_2$ for all $t \in (0,T)$. 
\end{proof} 

An immediate consequence of the previous result is that convex graphical MCF solutions can be smoothly approximated by compact ones. 
For any two compact convex hypersurfaces $\Sigma_1, \Sigma_2$ we write that  $\Sigma_1 \prec  \Sigma_2 $ if $\Sigma_2$ encloses $\Sigma_1$
(allowing  $\Sigma_1 \cap   \Sigma_2 \neq \emptyset$).

\begin{cor}\label{cor-approx}  Let   $M_t=
\{ (x, u(x,t)): \, x \in \R^n \} \subset \R^{n+1}$, $t \in (0, +\infty)$, be  a smooth  entire graph   Mean Curvature Flow solution    with initial data 
$M_0=
\{  (x, u_0(x)): \, x \in \R^n \} \subset \R^{n+1}$ which is a proper convex entire graph, normalized in such a way that  $u(0)=\min_{x \in \R^n} u_0(x) =0$.

Then, $M_t$ can be approximated by  a sequence $M^i_t$  of compact  convex Mean curvature flow solutions. 
More precisely,  the surfaces $\Sigma^i_t$  are  reflection symmetric with respect to the hyperplane $\{ x_{n+1} = i \}$ and  their lower parts  $\hat  \Sigma^i_t  := \Sigma^i_t \cap \{ x_{n+1}  < i \} $ converge, as $i \to +\infty$, to $M_t$, smoothly  on compact subsets of $\R^{n+1} \times (0, +\infty)$. 
\end{cor} 

\begin{proof}
From our assumptions we have  $M_t= \{ (x, u(x,t)) : x \in \R^n \}$, for all $t \in (0, +\infty)$ and that $u(\cdot, t) \geq   0$ for all $t \geq 0$, since 
 we have normalized our initial data so that $u(0) = \min_{ x \in \R^n} u_0(x) = 0.$ Furthermore, since $u_0(x)$ is assumed to be proper we have
 $\lim_{x \to +\infty} u(x,t) = +\infty$ for all $t \geq 0$. 

\smallskip 

For each integer  $i \geq  1$, we define the Lipschitz domains 
$$\cD^i_0 = \{ (x, x_{n+1} ) \in \R^{n+1} :   u_0(x) \leq x_{n+1} \leq 2i - u_0(x) \}$$
and we let  $\Sigma^i_0 = \partial \cD^i_0$. Our assumption that $u(0)=0$ guarantees that $\cD^i_0 \neq \emptyset$ for all $i \geq 1$.
Note that  $\Sigma^i_0 \subset  \R^{n+1}$ is just the closed  hypersurface that  consists by $M_0 \cap \{ x_{n+1} \leq i \}$
and its  reflection with respect to the hyperplane $x_{n+1} =i$. Furthermore, each $\Sigma_0^i$ is convex and Lipschitz continuous. 

\smallskip

Standard MCF theory shows that for any $i \geq 1$, there exists a unique smooth mean curvature flow  $\Sigma^i_t$ starting at $\Sigma_0^i$. 
The solutions $\Sigma^i_t$ exists up to times $T^i$, they satisfy   $\Sigma^i_t \prec \Sigma^{i+1}_t $ ($\Sigma^{i+1}_t$ encloses $\Sigma_t^i$),   and $\lim_{i \to +\infty} T^i =+\infty$. The strong maximum principle guarantees that
each $\Sigma_t^i$, $0 < t < T^i$  is strictly convex. 
Furthermore, $\Sigma^i_t$ is reflection symmetric with respect to the hyperplane $\{ x_{n+1}=i \}$, since $\Sigma_0^i$ is by construction.

Denote by  $\hat  \Sigma^i_t $
to be the lower half  of $\Sigma^i$, that is 
 $$\hat  \Sigma^i_t  := \Sigma^i_t \cap \{ x_{n+1}  < i \}.$$
Also, for any point $\bf{x_0}  \in \R^{n+1}$ let us  denote by $B_R^{n+1}(\bf{x_0})$ the ball in $\R^{n+1}$ of radius $R$ centered at ${\bf x_0}$. 

\begin{claim} Fix $T >0$.  For any  $R > 1$, there exists an integer $i_R$ such as long as  $i \geq i_R$,  the  lower part of 
$\hat \Sigma_t^i \cap B_{2R}^{n+1}({\bf 0}) $, $t \in [0,T]$  can be written as a graph $\big \{ (x, u^i(x,t))  \,: \, |x| \leq R  \big \}$   and   satisfies  a uniform in $i$ gradient bound  which is independent of $i$ and depends only on $R$ and $M_0$. 
\end{claim}

\begin{proof} Fix $T >0$ and assume that $i$ is chosen sufficiently large so that $T^i > T$. Furthermore, given any 
 $R >1$, we may choose  $i_R$ sufficiently large so that $T \ll R$ and if ${\bf x^i_0}=({\bf 0},i) \in \R^{n+1}$, then 
$B^{n+1}_{4R}({\bf x^i_0}) \prec \Sigma^i_t$,  for all $i \geq i_R$ and all $t \in [0,T]$.
%Furthermore, we may choose  $R$ sufficiently large so that the Mean curvature flow solution 
%$S_t$   starting  at $\partial B^{n+1}_{4R}(P^i_0)$ at time $t=0$,  encloses the ball  $B^{n+1}_{2R}(P^i_0)$ at time $t=T$. 
%Since $S_0=\partial B^{n+1}_{4R}(P^i_0)$
%in enclosed  by  $M_0^i$ for all $i \geq i_R$, the comparison principle
%guarantees that $S_t$ is enclosed by $\Sigma^i_t$ for all $t \in [0, T]$ and all $i \geq i_R$. Consequently,  $B^{n+1}_{2R}(P^i_0)$ is enclosed
%by $M_i^t$ for all $t \in [0,T]$, $i \geq i_R$. 

The  convexity and symmetry of the solutions $\Sigma^i_t$ then imply  that for any $i \geq i_R$,  $\hat \Sigma_t^i \cap B_{3R}^{n+1}({\bf 0}) $,  $t \in [0,T]$  can be written as a graph 
$\big \{ (x, u^i(x,t))  \,: \, |x| \leq 3R  \big \}$.  
Hence, it remains to show the uniform gradient bound of  $\hat \Sigma_t^i \cap B_{2R}^{n+1}({\bf 0})$, $t \in [0,T]$ for all $i \geq i_R$. 
This readily  follows from the local gradient bound in \cite{EH2}
and the fact that $u^i(x,0) = u_0(x)$ for all $i \geq i_R$,  which implies  that  $\Sigma_0^i \cap B_{3R}^{n+1}({\bf 0})$, $i \geq i_R$  satisfy a uniform gradient bound. 

%For any $R >1$, choose  $i_R$ sufficiently large so that 
%the ball 
%$B^{n+1}_{4R}(P_i) := \{ (x,x_{n+1}) \in \R^{n+1} : x^2 + (x_{n+1}-i)^2 \leq 4R \}$   of radius $4R$ centered at $P^i_0=(0,i)$ is enclosed by the 
%initial hypersurfaces $\Sigma^i_0$, for all $i \geq i_R$. Furthermore, we may choose  $R$ sufficiently large so that the Mean curvature flow solution 
%$S_t$   starting  at $\partial B^{n+1}_{4R}(P^i_0)$ at time $t=0$,  encloses the ball  $B^{n+1}_{2R}(P^i_0)$ at time $t=T$. 
%Since $S_0=\partial B^{n+1}_{4R}(P^i_0)$
%in enclosed  by  $M_0^i$ for all $i \geq i_R$, the comparison principle
%guarantees that $S_t$ is enclosed by $\Sigma^i_t$ for all $t \in [0, T]$ and all $i \geq i_R$. Consequently,  $B^{n+1}_{2R}(P^i_0)$ is enclosed
%by $M_i^t$ for all $t \in [0,T]$, $i \geq i_R$. 
%
%Let $Q^i=(x_0, u^i(x_0,t)) $ be any point in $\Sigma^i_t$ such that $|x| \leq R$ and set $P^i=(x_0, i)$. Then, $B^{n+1}_R(P^i) \subset B^{n+1}_{2R}(P^i) $ 
%which means that the ball $B^{n,i}_R(x_0): = B^{n+1}_R(P^i) \cap \{ x_{n+1} =i \}= \{ (x,i): |x-x_0| \leq R \}$ is enclosed by $M_i^i$. 

\end{proof} 

The results in \cite{EH2} then imply that $\hat \Sigma_t^i \cap B_{R}^{n+1}({\bf 0}) $, $t \in [0,T]$, $i \geq i_R$ 
have   uniformly bounded second fundamental forms. 
More precisely, there exists a constant $C_{R,T}$ that  is  independent of $i$  such that   the second fundamental form $|A^i|$ of $\Sigma^i$ satisfies the bound 
\be\label{eqn-AAA} \sup_{\hat \Sigma_t^i \cap B_R^{n+1}({\bf 0})} |A^i|   \leq C_{R,T} \, t^{-1/2}, \qquad t \in (0,T]
\ee
provided that $i \geq i_R$.  

One can then pass to the limit (over a subsequence
$i_k \to +\infty$)
and obtain a smooth entire graph  mean curvature flow  solution $\hat  M_t$, $t \in (0,T)$  whose second fundamental form satisfies the bound 
\be\label{eqn-AAA} \sup_{\hat M_t \cap B_R^{n+1}({\bf 0})} |A|   \leq C_{R,T} \, t^{-1/2}, \qquad t \in (0,T]. 
\ee
Standard arguments then imply  that if $\hat M_t = \{  (x, \hat u(x, t)): \, x \in \R^n \} $, then $\lim_{t \to 0} \hat  u(x,t) = u_0(x)$. 
Since $x_{n+1}=u_0(x)$ is proper,  $x_{n+1}=\hat u(x,t)$ is proper as well. 
Hence, Theorem  \ref{thm-general-convex} guarantees that  $u= \hat  u$ on $\R^n \times (0,T)$. Since $T >0$ was arbitrary,   we conclude that $u = \hat  u$
on $\R^n \times (0,+\infty)$ finishing the proof of the corollary. 

\smallskip

\end{proof} 

\begin{remark}
Our methods can be applied to study the uniqueness  of the   (convex) solutions that are analyzed by X.-J. Wang in \cite{XW}. More precisely,  in that paper, the author studies convex translating solutions to Mean Curvature flow via a level set method. In the non-compact case, those solutions are obtained via taking limits and our techniques can be used as an alternative proof of the uniqueness of such limits. We leave the details to the interested reader.
\end{remark}

\smallskip 

An immediate consequence of Corollary \ref{cor-approx}  is that Hamilton's  Harnack inequality holds  for entire convex graphs.

\begin{cor} [Hamilton's Harnack estimate]\label{cor:Harnack estimate}
Any smooth convex proper entire graph solution $M_t$, $t \in (0,+\infty) $  of Mean curvature flow satisfies Hamilton's Harnack differential inequality, namely for any tangent vector field $V$, 
\be\label{eqn-harnack}
\frac{\partial H}{\partial t} + 2 \langle \nabla H, V \rangle + h(V,V) + \frac H{2t} \geq 0.
\ee
\end{cor} 

\begin{proof} Let $\Sigma_i^t$ be approximating sequence of compact convex solutions which were constructed in Corollary \ref{cor-approx}.
Each of them satisfy the Harnack differential inequality \eqref{eqn-harnack}. Passing to the smooth limit on compact sets, we conclude that 
\eqref{eqn-harnack} also holds for our complete non-compact solution $M_t$, for all $t \in (0,+\infty)$. 

\smallskip

\end{proof}

\end{document}